\documentclass[preprint]{imsart}
\usepackage{latexsym}
\usepackage{hyperref}
\usepackage{amsthm,amssymb,amsmath,amsfonts}
\usepackage[latin1]{inputenc}
\usepackage{graphicx}
\usepackage[top = 3cm, bottom = 3cm, right = 3cm, left = 3cm]{geometry}

\theoremstyle{plain}
\newtheorem{theorem}{Theorem}[section]
\newtheorem{prop}[theorem]{Proposition}
\newtheorem{lemme}[theorem]{Lemma}
\newtheorem{corol}[theorem]{Corollary}

\newtheorem{hyp}[theorem]{Assumption}

\newtheorem{rque}[theorem]{Remark}
\newcommand{\be} {\begin{equation}}
\newcommand{\ee} {\end{equation}}
\newcommand{\bea} {\begin{eqnarray}}
\newcommand{\eea} {\end{eqnarray}}
\newcommand{\Bea} {\begin{eqnarray*}}
\newcommand{\Eea} {\end{eqnarray*}}

\numberwithin{equation}{section}

\def\P{\mathbb{P}}

\renewcommand\r{\mathbb{R}}
\newcommand{\E}{\mathbb{E}}

\newcommand\rr\right
\newcommand\un{\mathbf{1}}
\newcommand\uno[1]{\un_{\left\{#1\right\}}}

\newcommand{\N}{\mathcal{N}}
\newcommand\n{\mathbb{N}}

\newcommand{\R}{\mathbb{R}}

\newcommand{\eps}{\varepsilon}
\newcommand{\pro}[1]{\mathbb{P}\left(#1\right)}
\newcommand{\esp}[1]{\mathbb{E}\left[#1\right]}

\begin{document}

\begin{frontmatter}
	
\title{Sharp approximation and hitting times for stochastic invasion processes}
	
\runtitle{Sharp approximation for invasion processes}

\begin{aug}

\author{\fnms{Vincent} \snm{Bansaye}\thanksref{m4}\ead[label=e4]{vincent.bansaye@polytechnique.edu}},
\author{\fnms{Xavier} \snm{Erny}\thanksref{m4}\ead[label=e5]{xavier.erny@polytechnique.edu}}
\and \author{\fnms{Sylvie} \snm{M\'el\'eard}\thanksref{m4}\ead[label=e6]{xavier.erny@polytechnique.edu}}

\address{\thanksmark{m4} CMAP, CNRS, \'Ecole polytechnique, Institut Polytechnique de
Paris, 91120 Palaiseau, France}

\runauthor{V. Bansaye et al.}
\end{aug}



 \bigskip \emph{This article is dedicated to Francis Comets, whose mathematical curiosity and enthusiasm have always been a source of inspiration}.
 
\begin{abstract}
We are interested in the invasion phase for stochastic processes with interactions.  A single mutant with positive fitness arrives in a {large} resident population at equilibrium. By a now classical approach, the first stage of the invasion is well approximated by a branching process. The macroscopic phase, when the mutant population is of the same order as the resident population, is described by the limiting dynamical system.
We   capture   the intermediate mesoscopic phase for the invasive population and obtain sharp approximations. It allows us  to  describe the fluctuations of the hitting times of thresholds, 
which inherit a large variance from the first stage.
We apply our results to two models which are original motivations. {In particular, we quantify   the hitting times of critical values in cancer emergence and epidemics. }
\end{abstract}

\begin{keyword}[class=MSC]
        \kwd{37N25}
        \kwd{60F15}
        \kwd{60G55}
        \kwd{60H10}
\end{keyword}

 \begin{keyword}
 \kwd{Stochastic invasion process}
    \kwd{Branching process}
    \kwd{Dynamical system}
    \kwd{Stochastic processes approximation}
    \kwd{Hitting times approximation}
 \end{keyword}
\end{frontmatter}
  
\section{Introduction and main results}\label{sectionmodel}

We  aim to finely quantify a mutant invasion in a resident population at equilibrium. Such a situation is standard in eco-evolution, in cancer  emergence or in epidemiology when a single individual particularly well adapted  can develop its own subpopulation and invade the global population.   After a certain amount of time, the number of mutants becomes non negligible with respect to the resident population size, allowing to summarize the dynamics of  the stochastic population processes by their deterministic approximations. 
In Champagnat \cite{champagnat06}, a systematic approach was introduced to quantify the mutant invasion success, based on the  properties of the mutant birth and death process and its coupling with  branching processes  for which the survival probability was easily computed. That allowed to characterize the probability for the mutant process to attain a certain fixed threshold. This approach has been used by many authors, see e.g. \cite{champagnatmeleard2011, billiardferrieremeleardtran, baarbovierchampagnat, billiardsmadi,boviercoquilleneukirch}. 

To go in details in this invasion process, we introduce some scaling parameter~$K${, whose values are typically large,} characterizing the macroscopic population sizes and allowing to quantify each invasion step.
We consider a  bi-type birth and death process modeling the population size dynamics of  interacting resident and mutant individuals. 
We denote by  $N^K_R(t)$ (resp. $N^K_M(t)$) the number of resident individuals (resp. mutant individuals) at time~$t$ and the density of populations is defined by
$$X^K(t)=(X^K_R(t),X^K_M(t)) := (N^K_R(t)/K,N^K_M(t)/K).$$
The interaction is modeled through density dependence of individual birth and death rates, respectively $b_{\bullet}(X^K), \, d_{\bullet}(X^K)$ for $\bullet\in \{R,M\}$.
This dependence  may  model for instance  competition for resources, which   can make the death rates increase or the birth rates decrease.\\

 The processes $(N^K_{R}, N^K_{M})$   are   defined  on a probability space $\,(\Omega, {\cal F}, \mathbb{P})\,$, as strong  solutions (until potential explosion time) of  the stochastic differential systems driven by  Poisson point measures  $\N^b_R,\N^d_R,\N^b_M,\N^d_M$ with Lebesgue measure intensity on $\mathbb{R}_{+} \times \mathbb{R}_{+} $: 
\begin{align}
\label{def:stochproc}
N^K_\bullet(t)=N^K_\bullet(0) &+ \int_{[0,t]\times\r_+}\uno{u\leq N^K_\bullet(s-) b_\bullet(X^K(s-))}\N^b_\bullet(ds,du)\\
&-\int_{[0,t]\times\r_+}\uno{u\leq N^K_\bullet(s-) d_\bullet(X^K(s-))}\N^d_\bullet(ds,du),\nonumber
\end{align}
 for $\bullet\in\{R,M\}$, $K>0$ and $t\geq 0$. 
We assume  that $\N^b_R$ and $\N^d_R$ are independent and that $\N^b_M$ and $\N^d_M$ are independent too. We also assume that the Poisson point measures $(\N^b_R, \N^d_R, \N^b_M, \N^d_M)$ are independent of the initial condition $N^K(0)$. But $(\N^b_R,\N^d_R)$ and $(\N^b_M,\N^d_M)$ are not necessarily independent. We refer to Section~\ref{secECO} for an application in evolution where  all these Poisson point processes are independent  and to Section~\ref{secSIR} for an example in epidemiology where they are dependent.\\
The individual growth rates are defined on $\mathbb{R}_{+} \times \mathbb{R}_{+}$ by
$$F_{R}(x_R,x_M) = b_{R}(x_R,x_M) - d_{R}(x_R,x_M) ; \qquad F_{M}(x_R,x_M) = b_{M}(x_R,x_M) - d_{M}(x_R,x_M).$$ 

It is well known (cf \cite{K1,K2,ethierkurtz}) that under suitable assumptions on the parameters and when the initial conditions of the  process are proportional to $K$, the stochastic process $\,((X^K_R(t),X^K_M(t)): {t\geq 0})\,$ converges in probability{, and even almost surely and in~$L^2$} (when $K$ tends to infinity), on any finite time interval. The limit is the solution $(x_R(t),x_M(t))_{t\geq 0}$ of  the dynamical system
\begin{eqnarray}
\label{def:dynsys}
\begin{cases}
x_{R}'(t) &= x_{R}(t)\,F_{R}(x_{R}(t),x_{M}(t)) \\
x_{M}'(t) &= x_{M}(t)\,F_{M}(x_{R}(t),x_{M}(t)).
\end{cases}
\end{eqnarray}
We assume the existence of some (non-necessary unique {nor stable}) equilibrium $x^{\star}_{R}>0$ satisfying
$$F_{R}(x^{\star}_{R},0) = b_{R}(x^{\star}_{R},0) - d_{R}(x^{\star}_{R},0)=0.$$ 
In this work, we start from one mutant inside a resident population close to this equilibrium : 
\begin{equation}
\label{hitparade}
N^K(0)=(N_R^K(0), 1), \qquad N_R^K(0)\sim x_R^{\star} K.
\end{equation}
Thus $X^K(0)$ is close to $(x_R^{\star},0)$ and
the mutant process  $(N^K_{M}(t),t\ge 0)$ will firstly be close to the branching process $Z$  with individual birth and death rates respectively 
$$b_{\star} = b_{M}(x^{\star}_{R},0)  \ ; \ d_{\star} = d_{M}(x^{\star}_{R},0).$$ 
We assume that the initial growth rate of this process is positive
\begin{equation}
\label{hyp=growth}
r_{\star}= b_{\star}-d_{\star} = F_M(x^\star_R,0) >0.\end{equation}
This ensures that  the mutant population process  attains  values of order $K$ with positive probability,  at a time scale $\log K$. \\

Our ambition is    to finely  capture  the dynamics of the invasion process until  macroscopic levels, and to focus in particular on the intermediate scale, where  the mutant process  is large but negligible compared to $K$.
We consider  the hitting times of successive levels $n\in\n^\star$ of the mutant population process   when the initial condition is $N^K(0) = n_0  \in\n ^2$:
\begin{equation}
\label{temps-mut}
    T^{K}_{M}(n_0,n)=\inf\{t \geq 0 :N^K_{M}(t) \geq n\}=\inf\{t \geq 0 :X^K_{M}(t) \geq n/K\}.
    \end{equation}

Our first result  proves   that the mutant process is  equivalent to the branching process $Z$ 
before reaching  macroscopic level of order $K$, in the following sense, see Section \ref{sectionbranch}.  For any  $\xi_K\ll   K/(\log K)^{a} $  and $\eta>0$, we show that 
\begin{align*}
&\lim_{K\to +\infty} \mathbb{P}\bigg(\underset{t\leq T^K_{M}(N^K(0),\xi_K)}{\sup}~ \left|\frac{N^K_M(t)}{Z(t)}-1\rr| > \eta \, ;  \,  W>0\bigg)\\
&\qquad =\lim_{T\rightarrow \infty} \limsup_{K\rightarrow \infty} \mathbb{P}\bigg(\underset{T\leq t\leq T^K_{M}(N^K(0),\xi_K)}{\sup}~ \left|\frac{N^K_M(t)}{We^{r_{\star}t}}-1\rr| > \eta \, ;  \,  W>0\bigg)=0,
\end{align*}
where $$W=\lim_{t\rightarrow \infty} Z(t)e^{-r_{\star}t} \in [0,\infty) \qquad a.s.$$
is an exponential variable with an additional atom at $0$ corresponding to the extinction event (cf. Durrett \cite{durrett}). 
{The meaning of the above convergence is that, as long as the mutant population has not reached the threshold~$\xi_K$, the size of the population~$N^K_M(t)$ is equivalent in probability to the branching process~$Z(t)$ as~$K$ goes to infinity.} The value of $a$ is  $1$ or $2$ depending on the exponential stability of the equilibrium of the resident population alone, see  forthcoming Theorem~\ref{mainbranch} for a precise statement. These two cases correspond to a hyperbolic or a partially hyperbolic equilibrium. 

For the proof of the theorem, we use a coupling 
of the mutant process and the branching process and  we need to control the distance between them. We use that individuals  belonging to one of the two populations but not to the other in the coupling can  be seen as a branching structure with an additional immigration whose intensity increases with the value of $Z$ or $N_M^K$.

This approach is inspired by and complements previous works. 
Exact couplings have been proved both  in the context of invasion of mutants in evolutionary biology and in epidemiology, where the models lead to a resident population with a hyperbolic or a partially hyperbolic equilibrium.
  For the SIR model in continuous time, we  refer in particular to
 Ball and Donnelly \cite{BallDon} for  a proof of  exact coupling  before the population reaches the order of magnitude $\sqrt{K}$.   Barbour and al \cite{BHKK}  prove that exact coupling holds before the mutant population reaches order 
 {$K^{7/12}$} for multidimensional models in evolutionary biology. 
 Complementary results have been obtained for epidemiological models in \cite{BarbUtev, BR}. Our result allows to compare the original mutant process and the branching process until the mutant population reaches the order $K/(\log K)^a$.\\

When the  population of mutants  becomes large but is not yet macroscopic, we prove (cf Section~\ref{qttdyn}) that the deterministic approximation by the dynamical system is already valid in the following sense. Starting with a number of mutants  $1 \ll N^K_M(0)=K X^K_M(0) \ll K$,
	$$\underset{t \leq T^K_M(N^K(0),vK)}{\sup}\bigg\vert \frac{X^K_M(t)}{ x_M^K(t)}-1\bigg\vert
\underset{K\rightarrow\infty}{\longrightarrow}0, \hbox{ in probability},$$
where $v>0$,  and 
$(x^K_M(t), t\ge 0)$ is the solution of~\eqref{def:dynsys} with the initial condition $1/K \ll x_M^K(0)=X^K_M(0) \ll 1$ and $x^K_R(0)$ close to~$x_R^\star$.   We refer to Theorem \ref{maindynamic} for a precise statement which gives also the conditions on the initial resident population size and the admissible values of $v$ and quantitative estimates for the convergence. The boundary of $(\R_+)^2$ is invariant  for the dynamical system but this domain is unstable in the neighborhood of $\R_+\times \{0\}$ in the   invasion regime. This makes the problem delicate when controlling the stochastic flow. 
In the time window $[0,T^K_M(N^K(0),vK)]$ whose size is of order of magnitude $\log K$,  $X^K$ goes from $X^K_M(0) \ll 1$ to $v$.  This requires to deal with a long time scale and escape an unstable equilibrium, as in \cite{BHKK}. We consider in this paper different scales and techniques, which allow us to capture the  expected renormalization $X_M^K(t)/x_M^K(t)$ from low density  until any reachable macroscopic value.
\\

 The simulations of Figure~\ref{graph2} illustrate the different approximations corresponding to the SIR example of Section~\ref{secSIR}, both by the branching process (from the initial time) and by the dynamical system (from hitting times of large levels).
\begin{figure}
\hspace*{-0.6cm}\includegraphics[scale=0.37]{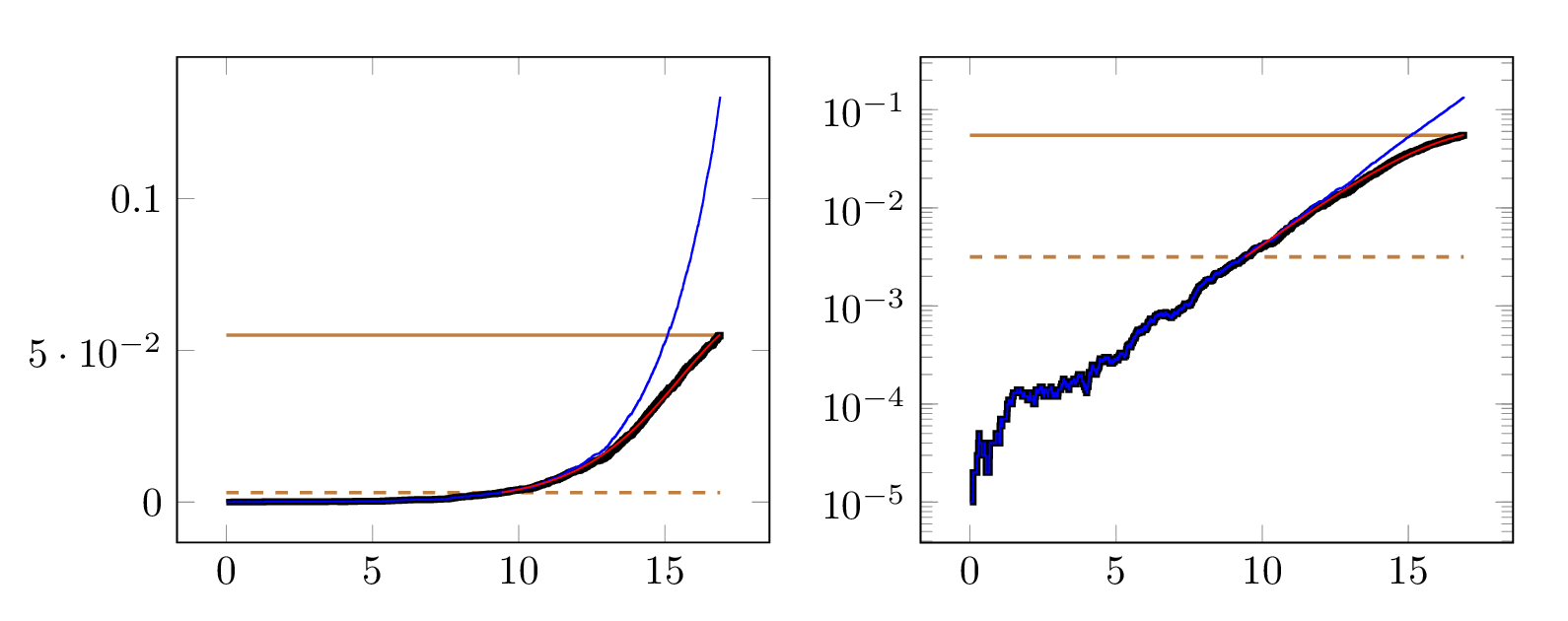}
\caption{\emph{The black curve represents a trajectory of the process $X^K_M(t)=N^K_M(t)/K$, the blue curve a trajectory of $Z(t)/K$, and the red curve is the deterministic function $x^K_M(t)$ starting at time $T^K_M(N^K(0),\xi^K)$. Both figures represent the simulations of the same trajectory in different scales (standard scale on the left, logarithmic scale on the right). These simulations have been done with the following parameters: $K=10^5$, $F_R(x_R,x_M)=-1.5 x_M$, $F_M(x_R,x_M)=1.5 x_R - 1$, $x_R^\star = 1$, $\xi_K=\sqrt{K}$ (cf the brown dashed line) and $v=0.055$ (cf the brown plain line).}}
\label{graph2}
\end{figure}
$\newline$

 As a byproduct,
we obtain in Theorem \ref{approxtps} the following convergence in law  of hitting times when $K$ tends to infinity, on the survival event of $N^K_M$. Under Condition \eqref{hitparade}, 
for $(\zeta_K)_K$ going to infinity such that $\zeta_K/K\rightarrow v \geq 0$, 
we show the following convergence in law
$$T_{M}^K(N^K(0),\zeta_K) - \frac{\log(\zeta_K)}{r_{\star}}\Longrightarrow \tau(v)-\frac{\log W^{\star}}{r_{\star}},$$
 where the function $\tau$ is  continuous, $\tau(0)=0$ and $W^{\star}$ is an exponential random variable with parameter~$r_{\star}/b_{\star}$.  We refer to Section \ref{defvstar} for the definition of~$\tau$ and 
 the characterization of admissible values $[0,v_{\star})$ of  the level $v$.
  This result is illustrated in  Figure~\ref{histo} corresponding to the SIR example of Section~\ref{secSIR}. Thus,  the hitting times of mesoscopic or macroscopic levels have  Gumbel laws translated by $(\log \zeta_K)/r_{\star}+\tau(v)$. We refer to \cite{BR,BHKK} for related results on the random time needed to see the deterministic  macroscopic curve.\\
In the situations when $K$ is very large, as for huge cohorts of microorganisms (bacteria, cells), $(\log \zeta_K)/r_{\star}$  gives a good approximation of the hitting times. But in many other cases, $K$ is not so large. We  refer to the modeling of hematopoiesis and leukemias \cite{bonnet1, bonnet2} where $K$ is the number of sain stem cells, of order $10^5$. This is the original motivation for this work {and one reason to quantify explicitly the convergence results}. The intermediate phase and hitting times are  indeed involved in the emergence and detection of cancer. 
 
\begin{center}
\begin{figure}
\includegraphics[scale=0.6]{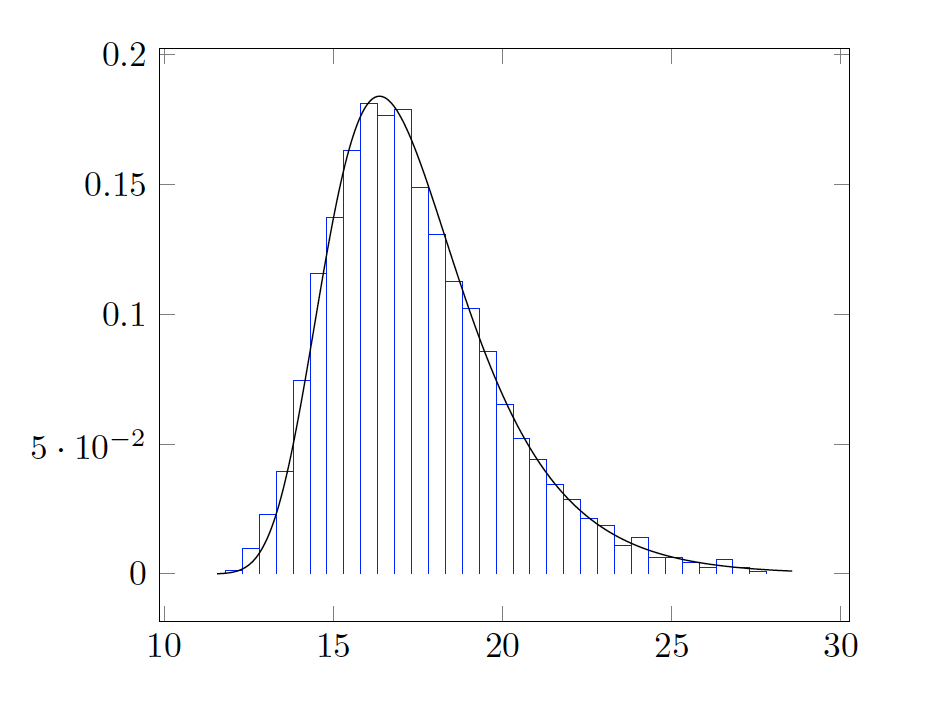}
\caption{\emph{ The histogram represents the values of $T_{M}^K(N^K(0),v K)$ (conditionally to the non-extinction of $N^K_M$). The simulations have been done with the following parameters: $K=100000$, $F_R(x_R,x_M)=-1.5 x_M$, $F_M(x_R,x_M)=1.5 x_R - 1$, $x_R^\star = 1$ and $v=0.055$. The size of the sample to draw the histogram is 10000.
The curve is the density $f$ of the random variable
 ${\log(vK)}/{r_{\star}} + \tau(v)-{\log(W^{\star})}/{r_{\star}}$,
 where the density of $-{\log(W^{\star})}/{r_{\star}}$ is equal to ${r_\star^2}/{b_\star}\exp(-r_\star t)\exp(- {r_\star}/{b_\star}e^{-r_\star t}).$
 }}
 \label{histo}
\end{figure}
\end{center}

\paragraph*{Assumptions.}

We consider the following assumptions on the birth and death rates. They ensure in particular that the resident population alone has a positive equilibrium point (which may be non stable) and that the mutant population can invade.  

\begin{hyp}\label{assumpt}$ $
	\begin{itemize}
		\item[$(R)$] Regularity: the functions $b_R,d_R,b_M,d_M$ are~$C^2(\mathbb{R}_{+} \times \mathbb{R}_{+}, \mathbb{R}_{+})$.
		\item[$(E)$] Equilibrium: there exists  $x^{\star}_R>0$ such that $F_R(x^{\star}_R,0)=0$ and  $\partial_R F_R(x^{\star}_R,0)\leq 0.$
		\item[$(I)$] Invasion: $F_M(x^{\star}_R,0)>0$. 
	\end{itemize}
\end{hyp}


\bigskip
Under  Assumption~$(R)$, the bi-type birth and death process $X^K$ is well defined until its explosion time, which may be finite with positive probability. In our framework, the process is studied before hitting times of fixed levels, clearly smaller than the explosion time. 
 In the eco-evolutionary framework,  $F_M(x^{\star}_R,0)$ is called invasion fitness. It quantifies the individual growth rate for a small mutant population in a resident population at the equilibrium.\\ 

\paragraph*{Notation.}
\begin{itemize}
\item We write $x^{\star}=(x^{\star}_R,0).$
\item For a smooth function $f:(x_R,x_M)\in\r^2\mapsto f(x_R,x_M)$, we denote by $\partial_R f$ (resp. $\partial_M f$) the derivative w.r.t. the first (resp. second) variable.
\item If $\N(dx)$ is a Poisson point measure with intensity $\lambda(dx)$, we denote by $\widetilde\N(dx) := \N(dx) - \lambda(dx)$ the compensated measure.
\item Throughout the paper, $C$ denotes any positive constant depending only on the model parameters. The value of~$C$ can change from line to line. If a constant depends on some (non-model) parameter~$\theta$ we write~$C_\theta$.
\end{itemize}

\section{Approximation by branching processes}\label{sectionbranch}
In this section,  we assume that at the initial time, the resident population is close to its equilibrium and a single mutant individual appears.
\begin{hyp}\label{assumptCI} There exists $C>0$ such that for any $K>0$,
\begin{equation*}
\bigg\vert \frac{N^K_R(0)}{K} - x_R^{\star} \bigg\vert \leq \frac{C}{K} \qquad \textrm{ and } \qquad N^K_M(0) = 1 \qquad \text{a.s}.
\end{equation*}
\end{hyp}

We introduce the branching process $(Z(t) : {t\geq 0})$  with individual birth rate $b_{\star}= b_M(x^{\star})$ and (individual) death rate $d_{\star}=d_M(x^{\star})$, defined on the same probability space as  $N^K_M$ and coupled with $N^K_M$ in the following way:
\begin{align*}
Z(t)=&1 + \int_{[0,t]\times\r_+}\uno{u\leq b_{\star}\,Z(s-) }\, \N^b_M(ds,du)
-\int_{[0,t]\times\r_+}\uno{u\leq  d_{\star}\, Z(s-)}\, \N^d_M(ds,du),
\end{align*}
where  the Poisson point measures $\N^b_M$ and $\N^d_M$ have been introduced in \eqref{def:stochproc}. 

\medskip We recall from Assumption~\ref{assumpt}.$(I)$ that   $r_{\star}=b_{\star}- d_{\star}>0$ and the branching process $Z$ is supercritical.
The survival event has positive probability and as  $t$ tends to infinity, the martingale $W(t)=Z(t)\exp(-r_{\star}t)$ converges a.s. to a finite random variable $W$ which is positive on the survival event :
$$\{W>0\}=\{\forall t>0 : Z(t)>0\}.$$ 
The random variable $W$  is exponentially distributed with an additional atom at $0$ :
 $$W\stackrel{d}{=}\frac{d_{\star}}{b_{\star}} \delta_0+ {\frac{r_{\star}}{b_{\star}}\, \mu_{ r_{\star}/b_{\star}}},$$
 where $\delta_0$ is a Dirac mass at $0$ and $\mu_a(dx)=\un_{\mathbb R_+}(x)ae^{-ax}dx$ and $dx$ is the Lebesgue measure.
 We refer
 to 
 \cite{AN} Chapter 3
 or \cite{durrett}-Theorem 1 for such results.

\medskip Let us study the coupling between the processes $Z$ and $N^K_M$.
First, when the branching process $Z$ becomes extinct, 
so does the mutant process $N_M$ and we have the following result.
 Under Assumptions~\ref{assumpt} and \ref{assumptCI}, 
	$$\lim_{K\rightarrow \infty} \P(\forall t\geq 0: N_M^K(t)=Z(t) \vert W=0)=1.$$
Indeed, the two processes coincide by coupling on a time window whose size goes to infinity with probability $1$ as $K$ goes to infinity. Besides, $Z$ is a.s. absorbed at $0$ in finite time on the event $\{W=0\}$. Such arguments  are classical for epidemiological models and can be adapted to our context. We refer to \cite{BallDon}.\\

Let us turn our attention to the survival event and state 
the main result of the section. It compares 
the mutant process and its branching approximation as long as the process is not too close to the macroscopic scale. We  prove that the ratio  of the two processes converges  to $1$ in probability, with an explicit speed, before time $T^K_{M}(N^K(0),n)$  when the population size of mutants reaches a certain threshold $n$ (conveniently chosen), see~\eqref{temps-mut} for definition.

\begin{theorem}\label{mainbranch}
	Under Assumptions~\ref{assumpt} and \ref{assumptCI}, 
 
	i)  if $\partial_R F_R(x^{\star})<0$,  there exists $C>0$ such that 
 for any $\eta>0$ small enough, $K\geq 2$ and $n\leq K/(\log K)$,
we have 
\begin{align*}
	&\mathbb{P}\bigg(\underset{t\leq T^K_{M}(N^K(0),n)}{\sup}~ \left|\frac{N^K_M(t)}{ Z(t)}-1 \rr|> \eta\ ; \ W>0\bigg) \leq C\left(\frac{\sqrt{K\log K}}{n} + \frac{1}{K^{1/3}} + \frac1{\eta^{1/4}}\left(\frac{n\log K}{K}\right)^{1/10}\right),
\end{align*}
	
	ii)  if $\partial_R F_R(x^{\star})=0$,
  there exist $L,C>0$ such that
	for any $\eta>0$ small enough, $K\geq 2$ and $n\leq  L K/(\log K)^2$, we have
	\begin{align*}
	&\mathbb{P}\bigg(\underset{t\leq T^K_{M}(N^K(0),n)}{\sup}~\left|\frac{N^K_M(t)}{ Z(t)}-1 \rr|>\eta\ ; \ W>0\bigg)
\leq C\left(\sqrt{\frac{\log K}{n}} + \frac{1}{K^{1/3}} + \frac1{\eta^{1/4}}\left(\log K\sqrt{\frac{n}{K}}\right)^{1/10}\right).
	\end{align*}
\end{theorem}

\medskip
{We observe that $n$ needs also to be larger that $\sqrt{K\log(K)}$ in $(i)$
(and larger than $\log(K)$ in $(ii)$) so that the right hand side tends to $0$. The choice of  $n$ 
allows  to optimize the bounds and  cover the time when the number  of mutants is below $K/\log K$ in (i) (and $LK/(\log K)^2$ in (ii)).
Let us state the immediate corollary which will be used in Section \ref{hittingtimes}.}

\begin{corol}
\label{mainbranch1}
If $\partial_R F_R(x^{\star})<0$ and $(\xi_K)_K$  satisfies
$1 \ll \xi_K\ll  K/(\log K)$, then
\begin{align*}
	\lim_{K\rightarrow \infty} &\mathbb{P}\bigg(\underset{t\leq T^K_{M}(N^K(0),\xi_K)}{\sup}~ \left|\frac{N^K_M(t)}{ Z(t)}-1 \rr|> \eta\ ; \ W>0\bigg)=0.
\end{align*}
If $\partial_R F_R(x^{\star})=0$, a similar result holds for sequences $(\xi_K)_K$ such that $1\ll\xi_K\ll K/(\log K)^2$.
\end{corol}

We refer to Section \ref{hittingtimes} for the asymptotic behavior of $T^K_M(N^K(0),\xi_K)$.

\medskip
 The statement depends on the value of $\partial_R F_R(x^{\star})$.   Indeed, the proof  is  based on a comparison between $T^K_{M}(N^K(0),n)$ and the first time where the resident population exits a neighborhood  of its equilibrium. If $\partial_R F_R(x^{\star})<0$, the  manifold $\{x_M=0\}$
 is stable around the fixed point $(x_R^{\star},0)$. This stability helps for the control of the resident population
 and guarantees that the distance of the  resident population process to its equilibrium  remains smaller than $n$ (up to some well chosen factor) until time $T^K_M(N^K(0),\xi_K)$. If  $\partial_R F_R(x^{\star})=0$, exponential (local) stability of residents is lost and we can (only) prove that the distance is smaller than $\sqrt{K n}$. We refer to Lemma~\ref{controlRbranch} for details.

\medskip The proof of Theorem~\ref{mainbranch} will be obtained in several steps that we develop below. We first introduce
a  coupling to compare $N^K_{M}$ and $Z$.  Let us introduce a decomposition involving the processes   realizing  the coupling and the additional part. The two subpopulations will be indexed respectively by $q$ and $r$.  \\
We introduce
$$m^K(s)=\min(N^K_{M,q}(s) b_M(X^K(s)),Z^K_q(s) b_{\star}), \quad p^K(s)=\max(N^K_{M,q}(s) d_M(X^K(s)),Z^K_q(s) d_{\star}),$$
where $(N^{K}_{M,q},N^{K}_{M,r})$ and $(Z^K_q,Z^K_r)$ are defined as follows:
\Bea
N^K_{M,q}(t)&=&1+\int_0^t\int_{\R_+}  \uno{u\leq m^K(s-)}\, \mathcal N^b_{M}(ds,du)-\int_0^t  \int_{\R_+} \uno{{0< u \leq p^K(s-)}} \,\mathcal N^d_{M}(ds,du)\\
N^K_{M,r}(t)&=&\int_0^t\int_{\R_+}  \uno{ m^K(s-)<u\leq N^K_{M,q}(s-) b_M(X^K(s-))}\, \mathcal N^b_{M}(ds,du)\\
&&\quad +\int_0^t  \int_{\R_+} \uno{{N^K_{M,q}(s-)d_M(X^K(s-)) < u \leq p^K(s-)}} \,\mathcal N^d_{M}(ds,du)\\
&&\quad +\int_0^t\int_{\R_+}  \uno{0\leq u- N^K_{M,q}(s-) b_M(X^K(s-))\leq N^K_{M,r}(s-) b_M(X^K(s-))}\, \mathcal N^b_{M}(ds,du)\\
&&\quad -\int_0^t  \int_{\R_+} \uno{{0\leq u - N^K_{M,q}(s-)d_M(X^K(s-))\leq N^K_{M,r}(s-) d_M(X^K(s-))}} \,\mathcal N^d_{M}(ds,du)
\Eea
and in a similar way,
\Bea
Z_q^K(t)&=&N^K_{M,q}(t)\\
Z_{r}^K(t)&=&\int_0^t\int_{\R_+}  \uno{ m^K(s-)<u\leq Z_q^K(s-) b_{\star}}\, \mathcal N^b_{M}(ds,du)\\
&&\quad +\int_0^t  \int_{\R_+} \uno{{Z^K_q(s-)d_\star < u  \leq p^K(s-)}} \,\mathcal N^d_{M}(ds,du)\\
&&\quad +\int_0^t\int_{\R_+}  \uno{0\leq u- Z_q^K(s-) b_{\star} \leq Z_{r}^K(s-) b_{\star}}\, \mathcal N^b_{M}(ds,du)\\
&&\quad -\int_0^t  \int_{\R_+} \uno{{0\leq u - Z^K_q(s-)d_\star\leq Z_{r}^K(s-) d_{\star}}} \,\mathcal N^d_{M}(ds,du).
\Eea
We can easily check that
\begin{lemme}
	\label{couplage} The following equalities of processes hold, almost surely: 
 $$N^K_M=N^K_{M,q}+N^K_{M,r}, \qquad Z=Z_q^K+Z_r^K.$$
\end{lemme}
\begin{proof} That  results from  the fact that  $N^K_{M,q}+N^K_{M,r}$ and   $N^K_M$ are solutions of the same stochastic differential equation for which pathwise uniqueness holds.  A similar argument also holds for $Z$.
\end{proof}

\bigskip Let us introduce 
the time when the gap to the equilibrium for resident population 
goes beyond some level $n\geq 0$ starting from the initial condition $N^K(0)=n_0\in \mathbb{N}^2$:
$$T^K_{R}(n_0, n)=\inf\left\{t>0~ : ~|N^K_R(t) - x^{\star}_RK|> n\rr\}=\inf\left\{t>0~: ~|X^K_R(t) - x^{\star}_R|> n/K\rr\}.$$
We consider  the (potentially random) initial population size $N^K(0)$ and  introduce the stopping time 
\begin{align}
\label{stoppingtime}
T^K_{R,M}(m,n) &= T^K_{M}(N^K(0),m) \wedge T^K_{R}(N^K(0),n) \wedge \frac{4 \log K}{3r_\star}.
\end{align}
The residual processes $N_r^K,Z_r^K$  will be compared to an inhomogeneous branching process and proved to be negligible until $T^K_{R,M}(m,n)$, with $m$ and $n$ suitably chosen.  
\begin{rque} This stopping time considers the first time when the population of mutants goes beyond $m$ or the population of residents has moved of $n$ from its equilibrium value. It is also truncated at a deterministic value of order $\log K$. 
For this truncation, we can consider any time $\lambda \,\log K/r_\star$ (with $\lambda>1$). We choose $\lambda := 4/3$ in definition \eqref{stoppingtime} to optimize the convergence speed of Theorem~\ref{mainbranch}. This particular choice only matters in the proof of Lemma~\ref{controlRbranch}.
\end{rque}

\medskip

\begin{lemme}\label{lemme26}
Under Assumptions~\ref{assumpt} and \ref{assumptCI}, 
\begin{itemize}
	\item[i)] For any~$L>0$, there exists $C_L>0$
 such that for any $n \leq K/(\log K)$, 
	\label{reste}
	$$\esp{\underset{t\leq T^K_{R,M}(Ln,n)}{\sup} (Z^K_r(t)+N^K_{M,r}(t))e^{-r_\star t}} 
	 \leq C_L\left(\frac{n\log K }{K}\right)^{1/2}\exp\left(C_L \frac{n\log K}{K}\right).$$
	
	
	\item[ii)] There exists $C> 0$ such that for any 
	$ n \leq K/(\log K)^2$,
	$$\esp{\underset{t\leq T^K_{R,M}({\sqrt{Kn}},n)}{\sup} (Z^K_r(t)+N^K_{M,r}(t))e^{-r_\star t}} 
\leq C\left(\log K\sqrt{\frac{n}{K}}\right)^{1/2} \exp\left(C \log K\sqrt{\frac{n}{K}}\right).$$
\end{itemize}
\end{lemme}

\medskip
\begin{proof} 
 Let us first prove i). Since $b_M$ and $d_M$ are locally Lipschitz, there exists $C_L>0$
 (which may change from line to line) such that  for any $s\leq T^K_{R,M}(Ln,n)$ and $n\leq K$,
	\begin{equation}\label{ZnKtaux}
	Z_q^K(s)\vert b_M(X^K(s))-b_{\star}\vert +Z_q^K(s)\vert d_M(X^K(s))- d_{\star} \vert\leq C_L Z_q^K(s)\frac{n}{K}
	\end{equation}
 almost surely. Then, recalling that $Z^K_q = N^K_{M,q}$,
	$$ Z_q^K(s) b_{\star} -m^K(s) \, + \,   p^K(s)-Z_q^K(s) d_{\star} \leq C_L Z_q^K(s)\frac{n}{K}.$$
	Therefore, the process 	$Z_r^K$ can be stochastically dominated 
	on the time window $[0, T^K_{R,M}(Ln,n)]$   by a branching process with inhomogeneous immigration $Y$, unique strong solution of
	\Bea
	Y(t)&=&\int_0^t\int_{\R_+}  \uno{ u\leq C_L Z_{q}^K(s-)\frac{n}{K} }\, \mathcal N^I(ds,du)\\
	&&\quad +\int_0^t\int_{\R_+}  \uno{0\leq u\leq  Y(s-) b_{\star}}\, {\mathcal N^b_Y(ds,du)} -\int_0^t  \int_{\R_+} \uno{0\leq u \leq Y(s-) d_{\star}} \,{\mathcal N^d_Y(ds,du)},
	\Eea
	{where $\mathcal N^I,\mathcal N^b_Y,\mathcal N^d_Y$ are independent Poisson point measures with Lebesgue intensity on~$\R_+^2$.} 
	We can easily note that for any $t\geq 0$, $$\esp{{Z_{q}^K}(t)}\leq \esp{Z(t)}=e^{r_{\star}t}$$   and straightforward computation leads to
	\begin{equation}\label{inegaliteY}
    \esp{Y(t\wedge T^K_{R,M}(Ln,n))} \leq C_L \, \frac{n}{K}\, e^{2r_{\star}t}.
    \end{equation}
    {Indeed, denoting $y(t) = \esp{Y(t\wedge T^K_{R,M}(Ln,n))}$, the function~$y$ satisfies the following~: for all~$t>0$,
    \begin{align*}
        y(t) =& C_L \frac{n}{K}\int_0^{t\wedge T^K_{R,M}(Ln,n)} \esp{Z^K_q(s)}ds + r_\star \int_0^{t\wedge T^K_{R,M}(Ln,n)} y(s)ds\\
        \leq& C_L\frac{n}{K} \int_0^t e^{r_\star s}ds + r_\star\int_0^t y(s)ds\\
        \leq& C_L\frac{n}{K} e^{r_\star t}/r_\star + r_\star\int_0^t y(s)ds.
    \end{align*}
    Hence Gronwall's inequality gives~\eqref{inegaliteY}.
    }
 
	Let $M({t})$ be the martingale part of the semimartingale $Y(t)$ {(i.e. defined by the same expression as~$Y$ replacing the three Poisson point measures by their compensated versions)}, with quadratic variation process given by 
	$$\langle M \rangle ({t}) = \int_{0}^t  \Big(C_LZ_{q}^K(s)\frac{n}{K} + Y(s) (b_{*}+ d_{*})\Big) ds.$$
Using Doob and Cauchy Schwarz inequalities and  $T^K_{R,M}(Ln,n )\leq
4(\log K)/(3 r_\star)$ by  \eqref{stoppingtime}, we obtain from the previous estimates 
	$$\E\left(\sup_{t\leq  T^K_{R,M}(Ln,n )}\left| \int_{0}^t e^{-r_{\star}s} dM({s})\right|\right)\leq 
	C \E\bigg(\int_{0}^{ T^K_{R,M}(Ln,n )} e^{-2 r_{\star}s}d \langle M\rangle(s)\bigg)^{1/2}\leq C_L\sqrt{\frac{ n \log K}{K}}.$$
We compute now
	\begin{align*}
	Y(t)e^{-r_{\star}t}=& \int_0^t e^{-r_\star s}dY(s) - r_\star\int_0^t Y(s) e^{-r_\star s}ds\\
    = & \int_{0}^t e^{-r_{\star}s} dM({s}) + \int_{0}^t e^{-r_{\star}s} \Big(C_LZ_{q}^K(s)\frac{n}{K} + Y(s) (b_{*}- d_{*}) -Y(s) (b_{*}- d_{*})\Big) ds\\
	= & \int_{0}^t e^{-r_{\star}s} dM({s}) + \int_{0}^t e^{-r_{\star}s} \,C_LZ_{q}^K(s)\frac{n}{K} ds.
	\end{align*}
Combining the previous estimates and using again $T^K_{R,M}(n,Ln )\leq
4\log K/(3r_{\star})$, we get 
	\begin{align*}
	\E\left(\sup_{t\leq T^K_{R,M}(Ln,n )} Y(t)e^{-r_{\star}t}\right)
	&\leq C_L \left( \sqrt{\frac{ n \log K}{K}} + \frac{ n \log K}{K}\right).
	\end{align*}
 Using the domination of $Z^K_{r}$ by $Y$, we obtain
	$$\E\left(\sup_{t\leq  T^K_{R,M}(Ln,n)} Z_r^K(t)e^{-r_{\star}t}\right)\leq  C_L\sqrt{ \log K \frac{n}{K} },$$
	which proves the first estimate 
 $i)$ for $Z_r$. \\
	
	The  assertion $i)$ for $N^K_{M,r}$ can be proved in a similar way. Indeed
	the individual birth and death rates are respectively upper-bounded by $b_{\star}+C_Ln/K$ and lower-bounded by $d_{\star}-C_Ln/K$ for $t\leq T^K_{R,M}(Ln,n)$.
	So $N^K_{M,r}$  is  dominated on the time window $[0,T^K_{R,M}(Ln,n)]$ by the process $\widetilde Y$
	defined by
	\begin{eqnarray*}
	\widetilde Y(t)&=&\int_0^t\int_{\R_+}  \uno{ u\leq C_LZ_q(s-)n/K }\, \mathcal N^I(ds,du)\\
	&&\quad +\int_0^t\int_{\R_+}  \uno{0\leq u\leq  \widetilde Y(s-) (b_{\star}+C_Ln/K)}\, \mathcal N^b_{\widetilde Y}(ds,du) -\int_0^t  \int_{\R_+} \uno{0\leq u \leq \widetilde Y(s-) (d_{\star}-C_Ln/K)} \,\mathcal N^d_{\widetilde Y}(ds,du).
	\end{eqnarray*}
 The conclusion follows as above by considering  $\widetilde Y(t)\exp(-(r_\star +2C_Ln/K)t)$ 
 for $t\leq T^K_{R,M}(Ln,n)$.\\

	For the second case $ii)$,  we proceed similarly. {The only difference is the upper-bound for the resident population in the stopping time~$T^K_{R,M}(\sqrt{Kn},n)$,}
 and~\eqref{ZnKtaux} becomes
 	$$Z_q^K(s)\vert b_M(X^K(s))-b_{\star}\vert +Z_q^K(s)\vert d_M(X^K(s))- d_{\star} \vert\leq C Z_q^K(s)\sqrt{\frac{n}{K}},$$
  almost surely     for $s\leq T^K_{R,M}(\sqrt{K n},n)$, since $n\leq K$. 
	Following the steps of case~$i)$,
	$$\E\left(\sup_{t\leq T^K_{R,M}(\sqrt{Kn},n)} Z_r^K(t)e^{-r_{\star}t}\right)\leq C \left(\log K \sqrt{\frac{n}{K}} \rr)^{1/2}.$$
	Proceeding similarly for $N_{M,r}^K$ ends the proof. {Notice that the exponential rescaling give the exponential factor in the statement of the lemma.}
\end{proof}

\medskip
{The next lemma guarantees that, in some sense, $N^K$ is equivalent to~$Z$ in probability until suitably chosen time~$T^K_{R,M}(m,n)$. The proof uses the fact that, by definition, $N^K_{M,q}$ and $Z^K_q$ are the same process, whence it is sufficient to control the residual processes~$N^K_{M,r}$ and~$Z^K_r$ normalized by~$Z$. The previous Lemma~\ref{lemme26} is important for this last step.}

\begin{lemme}\label{auxbranch}
	Under Assumptions~\ref{assumpt} and \ref{assumptCI}, 
	
	i) For any $L>0$, there exists  $C_L>0$
 such that for any small enough $\eta>0$ and $n\leq K/(\log K)$,
	\begin{align}
	\label{equivalent}
\mathbb{P}\bigg(\underset{t\leq T^K_{R,M}(Ln,n)}{\sup}~\left|\frac{N^K_M(t)}{Z(t)}-1\rr| > \eta\ ;\ W>0\bigg) \leq \frac{C_L}{\eta^{1/4}}\left(\frac{n\log K}{K}\right)^{1/10}\exp\left(C_L \frac{n\log K}{K}\right).
	\end{align}
	
	ii)  There exists $C>0$ such that for  any small enough $\eta>0$ and $n\leq K/(\log K)^2$,
		\begin{align}
		\label{equivalent2}
\mathbb{P}\bigg(\underset{t\leq T^K_{R,M}(\sqrt{Kn},n)}{\sup}~ \left|\frac{N^K_M(t)}{ Z(t)}-1\rr| > \eta\ ;\ W>0\bigg) \leq \frac{C}{\eta^{1/4}}\left(\log K\sqrt{\frac{n}{K}}\right)^{1/10}\exp\left(C\log K\sqrt{\frac{n}{K}}\right).
\end{align}
\end{lemme}
\begin{proof}
	 Let us prove $i)$ and focus on the  survival event $\{W>0\}$ where the process
	$Z$ stays positive.
	On this event, using Lemma  \ref{couplage}, we obtain 
	$$\frac{\vert N_M^K(t) -Z(t) \vert}{Z(t)} = \frac{\vert N_{M,r}^K(t) -Z^K_{r}(t) \vert}{  Z(t)}
 =\frac{\vert N_{M,r}^K(t) -Z^K_{r}(t) \vert e^{-r_{\star}t}}{  W(t)},$$
	where $W(t)= Z(t) e^{-r_\star t}$ is the classical martingale associated with $Z$.
	
	We use now Lemma~\ref{reste} $i)$ and the fact that $W(t)$ converges to $W$ a.s. as $t$ tends to infinity. We deduce that  for
	$t\leq  T^K_{R,M}(Ln,n)$,
	$ \vert N_{M,r}^K(t) -Z^K_{r}(t) \vert/ Z(t)$ tends to $0$ in probability on the event $W>0$, as $K$ tends to infinity.\\
	
	Let us now obtain   an explicit convergence speed. For convenience we define
	$$
 U^K = \underset{t\leq T^K_{R,M}(Ln,n)}{\sup} Z_r(t) e^{-r_\star t} + \underset{t\leq T^K_{R,M}(Ln,n)}{\sup} N^K_{M,r}(t) e^{-r_\star t}.$$
	The previous computations ensure that
		$$\frac{\vert N_M^K(t) -Z(t) \vert}{Z(t)} \leq \frac{  U^K}{\inf_{t\geq 0} W(t)}$$
	a.s.	on the event $W>0$.
Moreover, choosing $\varepsilon=n\log K/K$,  for any $\eta>0$,
	\begin{align*}
	   \pro{U^K \geq \eta~\underset{t\geq 0}{\inf}~W(t)\, ; \, W>0}
	    \leq
	    \pro{U^K\geq \varepsilon^{2/5}}+\pro{\underset{t\geq 0}{\inf}~W(t)\leq \frac1\eta \varepsilon^{2/5}\, ; \, W>0}.
	\end{align*}
	By Markov's inequality and Lemma~\ref{reste},
	$$\pro{U^K\geq \varepsilon^{2/5}} \leq C_L\,\varepsilon^{1/2-2/5} e^{C_L\varepsilon} = C_L \varepsilon^{1/10}e^{C_L\varepsilon}.$$
		On the other hand, by Lemma~\ref{infw} in Appendix,
	$$\pro{\underset{t\geq 0}{\inf}~W(t)\leq  \frac1\eta \varepsilon^{2/5}\, ; \, W>0} \leq C \frac1{\eta^{1/4}} \varepsilon^{2/5 \times 1/4} = C \frac{1}{\eta^{1/4}}\varepsilon^{1/10}.$$
	
	This proves the result in the  case~$i).$ Case~$ii)$ is proven similarly.
\end{proof}

\medskip Let us now compare the 
times when mutant and resident population processes reach their threshold.
\begin{lemme}\label{controlRbranch}
We denote $s_K=4 (\log K)/(3r_{\star})$.
	Under Assumptions~\ref{assumpt} and \ref{assumptCI},
	\begin{itemize}
	\item[i)] If $\partial_R F_R(x^{\star})<0$,
	then  there exist $L,C>0$ such that for any $K\geq 2$ and  $n\leq K/\log   K$,
	$$\pro{T^K_{M}(N^K(0),n)\geq T_R^K(N^K(0),Ln) \wedge s_K \, ; \, W>0 }\leq C\left(\frac{\sqrt{K\log K}}{n} +\frac{1}{K^{1/3}} + \left(\frac{n \log K}{K}\right)^{1/10}\right).$$
	
	\item[ii)] If $\partial_R F_R(x^{\star})=0$, then  
  there exist $L,C>0$ such that for any  $K\geq 2$ and  $n\leq L K/\log (K)^2$,
	$$\pro{T^K_{M}(N^K(0),n)\geq T_R^K(N^K(0),\sqrt{Kn}) \wedge s_K  \, ; \, W>0 }\leq C\left(\sqrt{\frac{\log K}{n}} +\frac{1}{K^{1/3}}+\left(\log K\sqrt{\frac{n}{K}}\right)^{1/10}\right).$$
 \end{itemize}
\end{lemme}
Observe that $\{T^K_{M}(N^K(0),n)\geq T_R^K(N^K(0),m) \wedge s_K\}=\{T^K_{R}(N^K(0),m) = T^K_{R,M}(m,n) \}$. 
We are using the  result for sequences with $m=\xi_K$ which makes the right hand side go to zero.
This will ensure that the mutant population process reaches its threshold before the resident population process.

\begin{proof}[Proof of case i)]
Let us write $T^K=T^K_{R,M}(Ln,n)$ for convenience. The first step consists in controlling the resident population. Let us    prove that there exist $L,C>0$
such that  for any $K\geq 2$ and $n\leq K/\log K$,
\begin{align}
&\mathbb P\left(T_R^K(N^K(0),Ln) \leq T^K_{M}(N^K(0),n)\wedge s_K\right) \nonumber \\
&\qquad \qquad =\pro{\underset{t\leq T^K }{\sup}~\left|X^K_R(t) - x^{\star}_R\rr| \geq L \frac{n}{K}} \leq C\frac{\sqrt{K \log K}}{n}, \label{STEP1}
\end{align}
For that purpose, we first notice that, for any $t\leq T^K$
and $K\geq 2$,
\begin{equation}
\label{borneprocess}
 X^K_M(t)\leq \frac{n}{K}\leq \frac{1}{\log K} \, \textrm{ and } \, X^K_R(t) \leq x^{\star}_R + L\frac{n}{K}\leq x^{\star}_R+\frac{L}{\log K}.   
\end{equation}
Hence, the processes $X^K_M$ and $X^K_R$ are bounded up to time~$T^K$, uniformly for $K\geq 2$.\\
We define  
\begin{align*}M^K(t)&=\frac1K\int_{[0,t]\times\r_+}\uno{u\leq N^K_R(s-) b_R(X^K(s-))}\widetilde\N^b_R(ds,du) \\
&\qquad - \frac1K\int_{[0,t]\times\r_+}\uno{u\leq N^K_R(s-) d_R(X^K(s-))}\widetilde\N^d_R(ds,du), 
\end{align*}
where 
 $\widetilde\N^b_R$ and $\widetilde\N^d_R$ are  the compensated Poisson martingale measures of $\N^b_R$ and $\N^d_R.$  We obtain
	\begin{align*}
X^K_R(t) 
	=& X^K_R(0)  + M^K(t) + \int_0^t G_R(X^K(s))ds,
	\end{align*}
 where 
	$$G_R(x_R,x_M) = x_R(b_R-d_R)(x_R,x_M) = x_R\,F_{R}(x_R,x_M).$$ 
	The quadratic variation process of the  martingale part~$M^K$ is the following  
\begin{align}	
\label{varqu} 
 \langle M^K\rangle({t}) = \frac{1}{K}\int_{0}^{t} X^K_R(s)\Big( b_R(X^K(s))+ d_R(X^K(s))\Big)ds.
 \end{align}
	Observing that $G_R(x^{\star})= 0$ by definition of $x^{\star}_R$, we rewrite the dynamics above as 
	\begin{align*}
	X^K_R(t) =& X^K_R(0)  + M^K(t) +H^K_{1}(t) + H^K_{2}(t)+\int_0^t (X^K_R(s) - x^{\star}_R) \partial_R G_R(x^{\star})ds,
	\end{align*}
	where
	$$H^K_{1}(t) =  \int_0^t X^K_M(s)\partial_M G_R(x^{\star})ds$$
	and
	$$H^K_{2}(t) = \int_0^t  \Big(G_R(X^K(s-)) - G_R(x^{\star}) - (X^K_R(s) - x^{\star}_R) \partial_R G_R(x^{\star}) - X^K_M(s)\partial_M G_R(x^{\star})\Big)ds.$$
Denoting
        $$Y^K_R(t) := X^K_R(t) - x^\star_R,$$
	we have
	\begin{equation}\label{YKRadrien}
	Y^K_R(t) = (X^K_R(0) - x^\star_R) +H^K(t)+ \int_0^t Y^K_R(s) \partial_R G_R(x^{\star})ds,
	\end{equation}
 where 
 $$H^K(t) := M^K(t)+H^K_{1}(t)+ H^K_{2}(t).$$
Recall that $\partial_R F_R(x^{\star})<0$ and so $\partial_R G_R(x^{\star})<0$.
Then Lemma~\ref{lemmeadrien} ensures that 
for any $K$ 
	\begin{eqnarray*}
		\underset{0\leq s\leq T^K}{\sup}\left|Y^K_R(s)\rr|&\leq& \Gamma \left|X^K_R(0) - x^\star_R\right| +\Gamma\, \underset{0\leq s\leq T^K}{\sup}~\left|H^K(s) - H^K(\lfloor s\rfloor)\rr|,\end{eqnarray*}
	with
	$$\Gamma := \left(1 + |\partial_R G_R(x^{\star})|\rr)\frac{1}{1 - e^{\partial_R G_R(x^{\star})}}.$$
Using that $H^K=M^K+H^K_{1}+ H^K_{2}$
and {(using Taylor-Lagrange's inequality to obtain the following bound for~$H^K_2$)}
$$\underset{t\leq T^K}{\sup} \left|H^K_1(t) - H^K_1(\lfloor t\rfloor)\rr| \leq \frac{n}{K} |\partial_M G_R(x^{\star})|, \quad \underset{t\leq T^K}{\sup} \left|H^K_2(t) - H^K_2(\lfloor t\rfloor)\rr|\leq  C\left(\frac{n}{K}\rr)^2$$
almost surely
for some constant $C>0$, by definition of $T^K$ and $H^K_1, H^K_2$, 
we get 
\begin{align*}
	\underset{0\leq s\leq T^K}{\sup}\left|Y^K_R(s)\rr|\leq  \Gamma\left(\underset{t\leq T^K}{\sup} \left|M^K(t) - M^K(\lfloor t\rfloor)\right|+ C\frac{n}{K} \right).
\end{align*}
Besides, using \eqref{borneprocess} and \eqref{varqu} and $T^K\leq C\log K$, we get
 $\esp{\langle M^K\rangle_{T^K}} \leq C (\log K)/{K}.$
 By Markov's and Burkholder-Davis-Gundy's inequalities, we then obtain
	\begin{align*}
	\pro{\underset{t\leq T^K}{\sup} \left|M^K(t) - M^K(\lfloor t\rfloor)\rr|\geq \frac{n}{K}}\leq& C  \frac{K}{n} \esp{\underset{t\leq T^K}{\sup}~|M^K(t)|}\\
	\leq&C  \frac{K}{n}  \esp{\langle M^K\rangle_{T^K}}^{1/2}\\
	\leq& C \frac{K}{n} \left(\frac{\log K}{K}\rr)^{1/2}\leq C\frac{\sqrt{K\log K}}{n}.
	\end{align*}
Recalling that   $Y^K_R(t)=X^K_R(t) - x^\star_R$  and $n\leq K$, {choosing~$L$ large enough} and combining these two last estimates  yields 
 the expected inequality $\eqref{STEP1}$
 and ends the first step.
 \\

 
We turn to the second step ii) and prove that 
	\begin{align}
	&\pro{T^K_{M}(N^K(0),n)\wedge T_R^K(N^K(0),Ln)\geq s_K \, ; \, W>0}  
 \leq  C K^{-1/3} +
 C\left(\frac{n\log K}{K}\right)^{1/10}.
 \label{tauMbranchement}
 \end{align}
Let us notice that, a.s. on the event~$\{W>0\},$
\begin{align*}
&\left\{T^K_M(N^K(0),n)\wedge T^K_R(N^K(0),n) \geq s_K\right\}\\
&\qquad \quad \subseteq\left(\left\{N^K_M(s_K)\leq n+1\right\}\cap \left\{\frac{Z(s_K)}{N^K_M(s_K)}\leq 2\right\} \right)\cup \left\{\underset{t\leq T^K}{\sup}\left|\frac{N^K_M(t)}{Z(t)}-1\right|>\frac12\right\}\\
&\qquad \quad \subseteq \left\{Z(s_K)\leq K+1\right\}\cup\left\{\underset{t\leq T^K}{\sup}\left|\frac{N^K_M(t)}{Z(t)}-1\right|>\frac12\right\}.
\end{align*}
Hence, using  $Z(s_K)=W(s_K)e^{r_\star s_K}=W(s_K)K^{4/3}$ (by definition of $s_K$),
\begin{multline*}
\pro{T^K_M(N^K(0),n)\wedge T^K_R(N^K(0),n) \geq s_K; W>0}\\
\leq \pro{W(s_K)\leq  2K^{-1/3}\, ; \, W>0} + \pro{\underset{t\leq T^K}{\sup}\left|\frac{N^K_M(t)}{Z(t)}-1\right|>\frac12 \, ; \, W>0}.
\end{multline*}
The first term  of the right hand side
involves an estimate on the infimum of a  single type branching process. It will be   stated and proved 
in Lemma~\ref{propbranchement} in forthcoming Appendix. The second term of the sum in the right hand side has been controlled in Lemma~\ref{auxbranch} $i)$. Observing that the exponential term in Lemma~\ref{auxbranch}
is bounded since $n\leq K/\log K$ and
combining these estimates yields
\eqref{tauMbranchement}.\\

We can now conclude the case $i)$
combining the two steps. More precisely
we use that for any $x,y,z\in\r,$ $(x\geq y\wedge z)\Longleftrightarrow (x\wedge y\geq z\textrm{ or }x\wedge z\geq y).$
The result follows then from  \eqref{tauMbranchement} and
  \eqref{STEP1}.
 \end{proof}

\begin{proof}[Proof of the case ii)]
 We prove a similar result under the hypothesis $\partial_R F_R(x^{\star})=0$. Let us prove
 that there exists $L>0$ such that for any $K\geq 1$ and $n\leq L.K/(\log K)^2$
	\begin{equation}
\label{wewant}	
 \pro{\underset{t\leq T^K}{\sup}~\left|X^K_R(t) - x^{\star}_R\right| \geq  \sqrt{n/K}}\leq C \sqrt{\frac{\log K}{n}},
 	\end{equation}
  where now	$T^K := T^K_{R,M}(\sqrt{n K},n)$.\\
  For that purpose,  we
use again \eqref{YKRadrien} and get  directly 
\begin{align}	
\label{inegggg}
 \underset{t\leq T^K}{\sup}~\left|X^K_R(t) - x^\star_R\right|=&\underset{s\leq T^K}{\sup}~|Y^K_R(t)|\\
 \leq& |X^K_R(0) - x^\star_R|+\underset{t\leq T^K}{\sup}~|M^K(t)|+\underset{t\leq T^K}{\sup}~|H^K_1(t)|+\underset{t\leq T^K}{\sup}~|H^K_2(t)|.\nonumber
 \end{align}
Recalling that $T_K\leq 4(\log K)/(3r_{\star})$ and proceeding as in the case $i)$, we obtain
	\begin{align}
	&\pro{\underset{t\leq T^K}{\sup} \left|M^K(t)\rr| \geq \frac{1}{2}\sqrt{\frac{n}{K}}} \leq  \ C\sqrt{\frac{K}{n}}\cdot \sqrt{\frac{\log K}{K}} =C \sqrt{\frac{\log K}{n}}, \label{etalorslamarinagle} \\
	&\underset{t\leq T^K}{\sup}~\left|H^K_1(t)\rr|+\underset{t\leq T^K}{\sup}~\left|H^K_2(t)\rr|\leq \  C \frac{n}{K} T^K 
 \leq  C \frac{n}{K}\log K= C \sqrt{\frac{n}{K}}\sqrt{\frac{n (\log K)^2}{K}}. \nonumber
	\end{align}
  For any  $n\leq K/(2C\log(K))^2$, we get
 $$\underset{t\leq T^K}{\sup}~\left|H^K_1(t)\rr|+\underset{t\leq T^K}{\sup}~\left|H^K_2(t)\rr|\leq \frac{1}{2} \sqrt{{\frac{n}{K}}}.$$
Then~\eqref{inegggg} guarantees
\begin{align*}
 \left\{\underset{t\leq T^K}{\sup}~\left|X^K_R(t) - x^{\star}_R\rr| \geq \sqrt{\frac{n}{K}}\right\}
 &\subset\left\{ \underset{t\leq T^K}{\sup}~|M^K(t)|\geq \frac{1}{2} 
\sqrt{\frac{n}{K}} \right\}
 \end{align*}
 and \eqref{etalorslamarinagle} yields \eqref{wewant} with $L=1/(4C^2)$. The proof is concluded as for 
 the case $i)$.
\end{proof}

\medskip
Theorem~\ref{mainbranch} becomes now a consequence of the two previous lemmas.
\begin{proof}[Proof of Theorem~\ref{mainbranch}]
We recall that $s_K=4(\log K)/(3r_\star)$ and  in the case $i)$ we  use
	\begin{align*}
	&\pro{\underset{t\leq T^K_M(N^K(0),n)}{\sup} \left|\frac{N^K_M(t)}{Z(t)}-1 \right| >\eta\, ; \, W>0}\\
	& \qquad \qquad \qquad \leq
 \pro{T^K_M(N^K(0),n)>T^K_R(N^K(0),Ln)\wedge s_K, W>0} \\
&\qquad \qquad \qquad \qquad \qquad \qquad	+ \pro{\underset{t\leq T^K_{R,M}(Ln,n)}{\sup}\left|\frac{N^K_M(t)}{Z(t)} -1 \right|>\eta\, ; \, W>0}.
	\end{align*}
The two terms of the right hand side  can be controlled using respectively Lemmas~\ref{controlRbranch} to choose $L$ and~\ref{auxbranch}. This completes the proof for $i)$, and $ii)$ is proved similarly.
 \end{proof}

\section{Approximation by dynamical systems}\label{sectiondyn}

The goal of this section is to approximate the process~
$X^K=(X^K_R,X^K_M)$ defined in Section~\ref{sectionmodel} by the dynamical system  defined by~\eqref{def:dynsys} under degenerate initial conditions. Usually such an approximation is proved when the initial conditions are of order of magnitude one (cf. for example  \cite{ethierkurtz}). Here, 
   we focus on  a small deterministic initial density  of mutants,  $1/K\ll x_M=N^K_M(0)/K\ll 1$ and an initial density of residents $x_R=N^K_R(0)/K$ close to the equilibrium value $x_R^\star$. The initial mutant density  may be very small regarding $K$ but we still enjoy a large number of mutants allowing to compare the stochastic process to its expected deterministic behavior.
Starting from such sub-macroscopic level of mutants,   the time for them to reach macroscopic levels
will be of order $\log K$.  We will compare the density  process and its deterministic approximation on this invasion time scale. They  are both small at the beginning and we will show that their ratio remains close to $1$ on the full time window allowing to reach macroscopic density of mutants.\\

This section emphasizes the role of (small) initial conditions and we introduce the flow notation  for convenience.
The process  $X^K(x,.)$ started from $x$  is defined 
 for any $K$ and time $t$ by
\begin{align}
X^K_\bullet(x,t)=&x_\bullet + \frac{1}{K}\int_{[0,t]\times\r_+}\uno{u\leq KX^K_\bullet(x,s-) b_\bullet(X^K(x,s-))}\N^b_\bullet(ds,du)\label{def:stochflow}\\
&\qquad - \frac{1}{K}\int_{[0,t]\times\r_+}\uno{u\leq KX^K_\bullet(x,s-) d_\bullet(X^K(x,s-))}\N^d_\bullet(ds,du)\nonumber
\end{align}
 for $\bullet\in\{R,M\}$.
The flow
$\phi(x,t)=(\phi_R(x,t),\phi_M(x,t))$ is the solution of \eqref{def:dynsys} starting from the initial value $x$, i.e. the unique solution of
$$ \phi(x,0)=x, \quad \frac{\partial}{\partial t} \phi(x,t)=G(\phi(x,t)), \quad \text{where} \quad G_{\bullet}(x)=x_{\bullet}(b_{\bullet}(x)-d_{\bullet}(x)).$$
In this section we  assume that  the  flow $\phi$ satisfies a boundedness assumption: the solution remains in a compact set if it starts close to the equilibrium $x^{\star}$.\\

\begin{hyp}
\label{hypdom}
There exists a compact domain $\mathcal{D}$  
of $\r_+^2$  
and $\varepsilon>0$ such that
$$   (x_R^{\star}-\varepsilon,x_R^{\star}+\varepsilon)\times [0,\varepsilon) \subset  \mathcal{D} \quad \text{ and }
\quad \{ \phi(x,t) : x \in \mathcal{D}, t\geq 0 \} \text{
is bounded.}$$
\end{hyp}



\subsection{Properties of the dynamical system and  hitting times}\label{defvstar}

Let us  consider the hitting time  for the mutant population starting from $x$
 $$\tau_M(x,v)=\inf\{ t\geq 0 :  \phi_M(x,t)=v\},$$
 where by convention $\inf \varnothing =\infty$ and in that case the hitting time is infinite.
 The following result quantifies the hitting time $\tau_M(x,v)$ as a function of the two variables~$x$ and~$v$, starting close to the equilibrium  $x^\star=(x^\star_R, 0)$.
 Let us define, for  $\eta>0,$
$$\mathcal C(x,\eta)=\sup\{ \phi_M(x,t)\,  : \, t\leq \overline{\tau}_M(x,\eta)\},$$
where $\overline{\tau}_M(x,\eta)$ is the first time when the growth rate of mutants is smaller than $\eta$:
$$\overline{\tau}_M(x,\eta)=\inf\{t \geq 0 : F_M(\phi(x,t))\leq \eta\}.$$
We set
$$v_\eta =\lim_{r \downarrow 0}\, \inf_{  x\in B_r(x^\star)\cap (\R_+^*)^2} \mathcal C(x,\eta),$$
where $B_r(x^\star)$ denotes the ball of radius $r$ centered in $x^\star$.
Observing that {$v_\eta$ is increasing as $\eta$ decreases}, we let $\eta$ go to zero and define    $v_{\star}$ as the upper-bound of the values 
of mutants which can be reached in the increasing phase, starting  from the neighborhood of the equilibrium:
$$v_\star := \lim_{\eta \downarrow 0}\,v_\eta.$$

The next result  guarantees that $v_\star$ is positive and finite. It also quantifies the hitting times $\tau_M(x,v)$ when $x$ goes close to $x^{\star}$ and $v$ is smaller than $v_{\star}$. 
 \begin{prop}\label{tKv} Under Assumptions~\ref{assumpt} and~\ref{hypdom},
the value $v_\star$ is positive and finite.\\
Besides there exists    a continuous increasing  function $\tau : [0,v_{\star})\rightarrow \R_+$ such that $\tau(0)=0$ and
for any $\bar v\in (0,v_{\star})$,
$$\lim_{ x \rightarrow x^{\star}} 
 \sup_{v\in (x_M,\bar v]}\big\vert \tau_M(x,v) - \frac{1}{r_{\star}}\log(v/x_M) -\tau(v)\big\vert=0,$$
 where the limit is taken for $x\in (\R_+^*)^2$ if $\partial_R F_R(x^{\star})<0,$ and for 
 $$x\in D(\eta) = \left\{(x_R,x_M) \in (\R_+^{*})^2~:~|x_R - x^{\star}_R|\log(1/x_M)\leq\eta\right\}$$
 for some $\eta>0$ if  $\partial_R F_R(x^{\star})=0$.
 \end{prop}

{The role of the set~$D(\eta)$ is to guarantee a suitable control of the resident population. Indeed, as $x_M$ decreases, the time to reach~$v$ increases (since it is of order~$\log(v/x_M)$), and so $x_R$ needs to be closer to~$x^\star_R$ to compensate it.}

The proof of this proposition necessitates some lemmas, which  guarantee the resident process
to remain close to its equilibrium value as long as 
the mutant process remains small. 



\medskip\begin{minipage}{\textwidth}
\begin{lemme}\label{xKRx*R}$ $
\begin{itemize}
\item[i)] If $\partial_R F_R(x^\star)<0,$ then there exist $u_0\in (0, 1)$ and $C>0$ such that
	for any $u \leq u_0$ and  $x = (x_R,x_M)\in B_u(x^\star)$,
	$$\underset{t\leq \tau_M(x,u)}{\sup}|\phi_R(x,t) - x^\star_R| \leq Cu.$$
\item[ii)]	If $\partial_R F_R(x^\star) = 0,$ then there exist  $u_0\in (0,1)$ and $\eta_0>0$ such that, for all $u\leq u_0$ and  $x\in {B}_u(x^\star)\cap D(\eta_0),$
	$$\underset{t\leq\tau_M(x,u)}{\sup}|\phi_R(x,t) - x^\star_R|\leq C u.$$
\item[iii)] In both cases, for all $u\leq u_0$ and  $x\in {B}_u(x^\star)\cap D(\eta_0),$
	$$\tau_M(x,u) \leq \frac{2}{r_\star}\log(1/x_M).$$
 \end{itemize}
\end{lemme}
\end{minipage}

\begin{proof}
For convenience we introduce the gap between the flow with its initial position for the resident population dynamics and the first time when the flow leaves~$B_r(x^{\star})$ (for any~$r>0$):
\begin{equation}\label{srx}
y_R(t) := \phi_R(x,t) - x_R, \qquad s_r(x):=\inf\{t \geq 0: \phi(x,t) \not\in  B_r(x^{\star})\}.
\end{equation}

{Note that, for the sake of notation, the initial condition~$x$ is not written explicitly in the quantity~$y_R(t)$ which still depends on it.}

To begin with, let us prove that, for~$r>0$ small enough,
\begin{equation}\label{srinflog}
    s_r(x)\leq \frac2{r_\star}\log(1/x_M).
\end{equation}

Recalling that $F_M(x^\star)=r_\star>0$, we have, for $r$ small enough and $t\leq s_r(x)$,
$$\partial_t\phi_M(x,t) =  \phi_M(x,t) F_M(\phi(x,t))\geq \phi_M(x,t)\frac{r_\star}{2}.$$

Whence, for all $t\leq s_r(x)$,
$$\phi_M(x,t) \geq x_M e^{t\cdot r_\star/2}.$$

The inequality above implies that $s_r(x)$ is finite. Then as $t\mapsto\phi_M(x,t)$ is continuous, by  choosing~$r<1$, we have
$$x_M e^{s_r(x)\cdot r_\star/2}\leq \phi_M(x,s_r(x)) < 1,$$
which implies
$$s_r(x)\leq \frac{2}{r_\star} \log(1/x_M).$$

Now let us prove~$i)$. As $\partial_R F_R(x^{\star})<0$, we can find $r>0$ such that for any $x\in B_r(x^{\star})$,
\begin{equation}
 \label{majorenegatif}
 \partial_R F_R(x)\leq \partial_R F_R(x^{\star})/2<0.
 \end{equation}

Let us use that by definition of the flow $\phi_R$
\begin{align}
\label{decompp}
y_R(t) = \int_0^t \phi_R(x,s)F_R(\phi(x,s))ds 
&=H(t) +\int_0^t \phi_R(x,s)y_R(s)\partial_R F_R(x_R,\phi_M(x,s))ds, 
\end{align}
where 
\begin{align*}
H(t)=\int_0^t \phi_R(x,s)(F_R(\phi(x,s))-y_R(s)\partial_R F_R(x_R,\phi_M(x,s)))ds.
\end{align*}
Recalling Assumption \ref{assumpt} (R) and
using Taylor expansion, there exists $C_r>0$ such that 
$$\big\vert F_R(\phi(x,s))-F_R(x_R,\phi_M(x,s))-y_R(s)\partial_R F_R(x_R,\phi_M(x,s))\big\vert \leq 
C_r \vert y_R(s)\vert^2\  $$
for $s\leq s_r(x)$. Using that $F_R(x_{\star})=0$, we get  also
$$\vert F_R(x_R,\phi_M(x,s))\vert { = \vert F_R(x_R,\phi_M(x,s)) - F_R(x^\star_R,0)\vert} \leq C_r \left(\vert x_R-x_R^{\star}\vert +\vert \phi_M(x,s)\vert\right)$$
and, combining these two bounds,
$$\sup_{t\leq T\wedge s_r(x)} \vert H(t)-H(\lfloor t\rfloor)\vert \leq C_r \left(\vert x_R-x_R^{\star}\vert+  \sup_{t\leq T\wedge s_r(x)}\left(\vert \phi_M(x,s)\vert+ \vert y_R(s)\vert^2\right)\right).$$
Thanks to~\eqref{majorenegatif}, we can use  Lemma~\ref{lemmeadrien} for \eqref{decompp}
and we get, for any $T>0,$
\begin{align}
\underset{t\leq T\wedge s_r(x)}{\sup}\left|y_R(t)\right| \leq& C_r\left( \underset{t\leq s_r(x)\wedge T}{\sup}\left|H(t) - H(\lfloor t\rfloor)\right| \right)\label{step11}\\
\leq& C_r \left(\vert x_R-x_R^{\star}\vert+  \sup_{t\leq T\wedge s_r(x)}\left(\vert \phi_M(x,s)\vert+ \vert y_R(s)\vert^2\right)\right). \nonumber
\end{align}
{In the rest of the proof, the value of~$C_r$ is fixed.} We work now for times when $y_R$ is also smaller that $1/(2C_r)$ and introduce
$$\overline{s}_r(x)=\inf\left\{t \geq 0 : \vert y_R(s)\vert\geq \frac{1}{2C_r}\right\} \wedge s_r(x).$$
Then for $s\leq \overline{s}_r(x)$,  this yields $C_r\vert y_R(s)\vert^2\leq \vert y_R(s)\vert/2$ and gathering terms with $y_R$,
\begin{align}
\underset{t\leq T\wedge \overline{s}_r(x)}{\sup}\left|y_R(t)\right| 
\leq& 2C_r \left(\vert x_R-x_R^{\star}\vert+ \sup_{t\leq T\wedge s_r(x)}\vert \phi_M(x,s)\vert \right).\nonumber
\end{align}
For any $u>0$ and  $x\in B_{u}(x^{\star}),$
\begin{equation}
\label{blabla}
 \underset{t\leq  \tau_M(x,u)\wedge \overline{s}_r(x)}{\sup}\left|y_R(t)\right| \leq  4C_r u.
 \end{equation}
Thus, setting $u_0=(r/2)\wedge r/(8C_r)\wedge 1/(16 C_r^2)$ ensures that for any $x\in B_{u_0}(x^{\star}),$
$$ \underset{t\leq  \tau_M(x,u_0)\wedge \overline{s}_r(x)}{\sup}\left|y_R(t)\right|\leq (r/2)\wedge 1/(4C_r),$$
which guarantees 
$ \overline{s}_r(x)>\tau_M(x,u_0)$.  Thus for  $u\leq u_0$, one has $ \overline{s}_r(x)>\tau_M(x,u)$ and 
\eqref{blabla} becomes
\begin{align}
\underset{t\leq  \tau_M(x, u)}{\sup}\left|y_R(t)\right| 
\leq& 
 4C_ru,\nonumber
\end{align}
since $\tau_M(x,u)\leq \tau_M(x,u_0)< \overline{s}_r(x)$. \\

 We prove now $(ii)$ and consider the case $\partial_R F_R(x^{\star})=0.$
 Recalling that the flow is bounded from Assumption~\ref{hypdom}, 
a Taylor expansion in   $x^{\star}$ yields now
 $\vert F_R(\phi(x,t)) \vert =  \vert F_R(\phi(x,t))-F_R(x^{\star})\vert 
 \leq  C(\vert \phi_R(x,t)-x_R^{\star}\vert^2+\vert \phi_M(x,t)\vert)$ and recalling that
 $y_R(t) = \phi_R(x,t) - x_R$,
\begin{align}
|y_R'(t)|= \vert \phi_R(x,t)F_R(\phi(x,t))\vert
\leq& C\left[|x_R - x^{\star}_R|^2 +y_R(t)^2+\phi_M(x,t)\right].\label{derivy}
\end{align}
Let us now introduce the first time when the quadratic contribution  may become dominant :
$$s(x)=\inf\{t \geq 0 : y_R(t)^2 \geq  |x_R - x^{\star}_R|^2 +\phi_M(x,t)\}.$$
Before this time, we can use
$|y_R'(t)|\leq 2C\left[|x_R - x^{\star}_R|^2 +\phi_M(x,t)\rr]$ and we get for $t\leq s(x)$,
\begin{align*}
|y_R(t)|=\left| \int_0^t y'_R(s)\right| 
&\leq C 
\left[|x_R - x^{\star}_R|^2 t+ 
\int_0^{t} \phi_M(x,s) ds\rr].
\end{align*}

Let us use again the notation $s_r(x)$ introduced at~\eqref{srx}, and fix $r>0$ to some small enough value such that for $t\leq s_r(x)$, $F_M(\phi(x,t))\geq r_\star/2$. Hence, for $t\leq s_r(x)$,
\begin{align}
\label{intphi}
\int_{0}^{t} \phi_M(x,s)ds &= \int_0^{t} \partial_s\phi_M(x,s)\frac{1}{F_M(\phi(x,s))}ds \leq \frac{2}{r_{\star}} \int_0^{t}\partial_s\phi_M(x,s)ds \leq \frac{2}{r_{\star}} \phi_M(x,t).
\end{align}
 We obtain for $t\leq s(x)\wedge s_r(x)$ (using~\eqref{srinflog}),
\begin{align}
\label{majy}
|y_R(t)|=\big\vert \int_0^t y_R'(s)\big\vert 
&\leq C 
\left[|x_R - x^{\star}_R|^2 \, \log(1/x_M)+ 
 \phi_M(x,t) \rr].
\end{align}
{From now the value of~$C$ is fixed and assumed to be greater than one (possibly by increasing it) and we consider $\eta_0:=1/(4C^2)$}. Let $x\in D:=D(1/(4C^2))=\{(x_R,x_M) : |x_R - x^{\star}_R| \, \log(1/x_M)\leq 1/(4C^2)\}$ and $t\leq \tau_M(x,1/(4C^2))\wedge s(x)\wedge s_r(x)$ such that
$\phi_M(x,t)\leq 1/(4C^2)$. We get
\begin{align*}
|y_R(t)|^2
&\leq 2C^2 
\left[(|x_R - x^{\star}_R| \, \log(1/x_M)^2|x_R - x^{\star}_R|^2+  \phi_M(x,t)^2
 \rr]
 \leq \frac{1}{2}\left[|x_R - x^{\star}_R|^2+  \phi_M(x,t)
 \rr].
\end{align*}
This implies that for any $x\in D$, $s(x)>\tau_M(x,1/(4C^2))\wedge s_r(x)$ and~\eqref{majy}
becomes
\begin{align*}
\sup_{t\leq \tau_M(x,1/(4C^2))\wedge s_r(x)}|y_R(t)|
&\leq C 
\left[|x_R - x^{\star}_R|+ 
 \phi_M(x,t) \right].
\end{align*}

Finally choosing some $u_0\leq 1/(4C^2)\wedge r/(4C),$ the inequality above implies that $\tau_M(x,u_0) < s_r(x)$, which proves the result.\\

{\it Proof of iii).} A straightforward consequence of~$i)$ and~$ii)$ is that, in both cases, it is possible to define, for any $r_0>0$ small enough, some~$u_0>0$ small enough such that for all~$u\leq u_0$,
$$\tau_M(x,u)\leq s_{r_0}(x).$$
Then~\eqref{srinflog} gives the result.
\end{proof}

{The following lemma is not used in the proof of Proposition~\ref{tKv}. However, since it is similar to the previous one (guaranteeing a control of the cumulated gap between the resident population dynamics and its equilibrium), we state and prove it here.}
\begin{lemme}\label{xKRx*Rdeux}
i)	Assume that $\partial_R F_R(x^\star)<0.$ Then there exist  $u_0\in (0, 1)$ and $\eta_0>0$ such that for all $u\leq u_0,\eta\leq\eta_0$ and  $x\in {B}_u(x^\star)\cap D(\eta),$
	$$\int_0^{\tau_M(x,u)}\underset{s\leq t}{\sup}~|\phi_R(x,s) - x^\star_R| \, dt \leq C(u + \eta).$$
ii)	Assume that $\partial_R F_R(x^\star)=0.$ Then there exist some $u_0\in (0,1)$ and $\eta_0>0$ such that for all $u \leq u_0$, and $\eta\leq\eta_0$ and  $x\in {B}_u(x^\star)\cap D(\eta),$
	$$\int_0^{\tau_M(x,u)}\underset{s\leq t}{\sup}~|\phi_R(x,s) - x^\star_R| \, dt \leq C(u + \eta^2).$$
\end{lemme}

\begin{proof}
Here again,  $x\in B_u(x^\star)$ and we will choose $u$ small enough to have the following estimates. 

 Firstly, let us prove the result under the assumption $\partial_R F_R(x^{\star})<0$. Using ~\eqref{step11} and noting
 $$z(t)=\underset{s\leq t}{\sup}~|y_R(s)|$$
 we have for all $t\leq \tau_M(x,u),$
 \begin{align*}
z(t) 
\leq& 
C\left( \vert x_{R}-x_R^{\star}\vert+\phi_M(x,t) + z(t)^2\right),
\end{align*}
since the function $t\mapsto \phi_M(x,t)$ is non-decreasing on $[0,\tau_M(x,u)]$ for $u$ small enough. As a consequence, 
\begin{align*}
\int_0^{\tau_M(x,u)} z(s) \,ds 
\leq&
C\bigg( \tau_M(x,u)|x_R - x^{\star}_R|+
\int_0^{\tau_M(x,u)}\phi_M(x,s)ds+z(\tau_M(x,u)) \int_0^{\tau_M(x,u)} z(s)ds\bigg).
\end{align*}
We use  Lemma~\ref{xKRx*R} to estimate  $z(\tau_M(x,u))$. Moreover
 \eqref{intphi} gives a bound for $\int_0^{\tau_M(x,u)}\phi_M(x,s)ds$. Adding that  $\tau_M(x,u)\leq 2\log(1/x_M)/r_\star$ (see Lemma \ref{xKRx*R} iii)), we obtain
 for any $x\in D(\eta)$,
$$\int_0^{{\tau_M(x,u)}}  z(s)\, ds\leq C \, \left( \eta \, + \,  u\,  + u\,\int_0^{\tau_M(x,u)}  z(s)\, ds\right).$$
Then for $u \in (0,1/(2C)]$, for $x\in D(\eta)$, 
$$\int_0^{{\tau_M(x,u)}}  z(s)  ds \leq 2C (\eta +u).$$
We conclude using that $\ \phi_R(x,s)-x_R^{\star}=y_R(s)+x_{R}-x_R^{\star}\ $.

\smallskip
 We can proceed similarly  under the assumption $\partial_R F_R(x^{\star})=0,$ using now~\eqref{derivy} 
 and $$\int_0^T\int_0^t z(s)^2dsdt
\leq \int_0^T z(t)\int_0^t z(s)dsdt
\leq \left(\int_0^T z(t)dt\right)^2.$$
The proof is completed.
\end{proof}

{The next lemma allows to lower-bound the growth rate of the mutant population on the time-windows we are interested in. It is used to quantify the hitting times~$\tau_M(x,v)$.}

\begin{lemme}\label{lemmeetav}
{The value of~$v_\star$ is positive and finite.} For all $v\in (0,v_\star)$, there exist ~$\eta\in (0,  r_\star)$ and $r>0$ such that, for all $x\in B_{r}(x^\star)\cap (\r_+^*)^2$ and  $t\leq \tau_M(x,v), \, F_M(\phi(x,t))\geq \eta.$
\end{lemme}

\begin{proof}
Recall that $r_{\star}=F_M(x^\star)$. We consider the case $\partial_R F_R(x^{\star})<0$ and the other case can be treated similarly. Thanks to Lemma~\ref{xKRx*R}, 
we can choose   $u_0>0$ such that, for all $u\leq u_0,$ and $x\in B_u(x^\star),$
\begin{align*}
\left|\phi_R(x,\tau_M(x,u)) - x^\star_R\right|\leq \underset{t\leq \tau_M(x,u)}{\sup}\left|\phi_R(x,t) - x^\star_R\right|\leq Cu.
\end{align*}
Then
$$\{\phi(x,t) : t\leq \tau_M(x,u)\} \subset ([x^\star_R-Cu,x^\star_R + Cu]\times[0,u])\cap (\r_+^*)^2.$$
Since $F_M(x^{\star})>0$, we can  choose $u_1\in (0,u_0]$  such that
\begin{equation}\label{FRlb}
\eta :=\inf \, \{F_M(\phi(x,t)) : t\leq \tau_M(x,u_1),x\in B_{u_1}(x^\star)\}>0.
\end{equation}

This ensures that $v_{\star}\geq u_1>0$.

Besides, starting close to the neighborhood of $x^{\star}$, the flow is bounded by  Assumption \ref{hypdom}. This ensures that $v_{\star}$  is finite.\\

Note that~\eqref{FRlb} only proves the second statement of the lemma for arbitrary small values of~$v<v_\star$. For the general case, consider some~$v<v_\star$. Then there exists  $\eta\in (0, r_\star)$ such that $v< v_{\eta}.$
    Hence there exists some $r>0$ such that
    $$v<\inf_{ x\in B_{r}(x^\star)\cap (\R_+^*)^2} \mathcal C(x,\eta).$$
Let $x$ belong to $B_{r}(x^\star)\cap (\R_+^*)^2.$ Then
   $$v< \mathcal C(x,\eta)=\sup\{ \phi_M(x,t)\,  : \, t\leq \overline{\tau}_M(x,\eta)\}$$ and we get $F_M(\phi(x,t))\geq \eta$ for $t\leq \tau_M(x,v)$ by continuity   of $t\rightarrow F_M(\phi(x,t))$.
\end{proof}
\medskip
We can now prove Proposition \ref{tKv}.
\begin{proof}[Proof of Proposition \ref{tKv}]
Let us consider the case~$\partial_R F_R(x^\star_R)<0$. We start by proving the result for small values of~$v<v_\star$.

By respectively Lemmas~\ref{xKRx*R} and~\ref{lemmeetav}, there exist some $0<u_1<v_\star$ and $\eta>0,$ such that, for all $u\leq u_1$ and $x\in B_{u}(x^\star),$
\begin{align}
\label{noel}
\left|\phi_R(x,\tau_M(x,u)) - x^\star_R\right|\leq \underset{t\leq \tau_M(x,u)}{\sup}\left|\phi_R(x,t) - x^\star_R\right|\leq Cu,
\end{align}
and for all $t\leq \tau_M(x,u),$
$$F_M(\phi(x,t))\geq\eta.$$

Furthermore
$$\tau_M(x,\phi_M(x,t))=t$$
and  
$$\partial_t \phi_M(x,t) = \phi_M(x,t) F_M( \phi_R(x, \tau_M(x,\phi_M(x,t))),\phi_M(x,t)).$$
By  separation of variables  (put $u=\phi_M(x,t)$) and integration on the time interval $[0,\tau_M(x,v)]$, we get
\begin{align*}
\tau_M(x,v) 
&= \int_{x_M}^{v} \frac{1}{u F_M( \phi_R(x, \tau_M(x,u)), u)}du 
= \frac{1}{r_{\star}}\log(v/x_M) + \mathcal R_v(x),
\end{align*}
for any $v\in (0, u_1]$ and
$x\in B_{u_1}(x^\star)$ (with $x_M\leq v$) and $t\leq \tau_M(x,v)$,
where 
\begin{align*}
\mathcal R_v(x)=\int_{x_M}^v\frac1u \left(\frac{1}{F_M(\phi_R(x, \tau_M(x,u)), u))}-\frac{1}{r_{\star}}\rr)du.
\end{align*}
To get the expected estimate, we need to evaluate 
the term inside the integral for  $u$ close to $0$.
Using that $F_M$ is locally Lipschitz and $r_{\star}=F_M(x^\star)$ and  
\eqref{noel},
and we obtain that for any $u\in (0, u_1]$ and
$x\in B_{u_1}(x^\star)$ (with $x_M\leq u$),
$$\frac1u \left(\frac{1}{F_M(\phi_R(x, \tau_M(x,u)), u))}-\frac{1}{r_{\star}}\rr)\leq C.$$
The integral of the right hand side is convergent at~$0$ and we get
$$\lim_{v\rightarrow 0} \, \sup_{x\in B_{u_1}(x^\star) : x_M\leq v} 
\, \mathcal R_v(x)
=0.$$
 Besides, for any $u_2\in (0,u_1]$,
$(x,u) \rightarrow F_M(\phi_R(x, \tau_M(x,u)), u)$ is uniformly continuous on 
$\{(x,u) :  x\in B_{u_1}(x^\star), x_M\leq u, u\in [u_2,u_1]\}$. This can be shown 
by a linearization argument at the equilibrium $x^{\star}$ {using Hartman-Grobman's Theorem (see e.g. Section~3.2 of \cite[p. 24]{hofbauer} and particularly Theorem~3.2.1)}.
Combining both facts and splitting the interval  $[x_m,v]=[x_m,u_2]\cup [u_2,v]$ yield 
$$\sup\left\{\left\vert \mathcal R_v(x)-\mathcal R_v(x')\right\vert
:  x,x' \in B_r(x^{\star})\cap (\R_+^*)^2, v \in [\max(x_M,x_M'), u_1]\right\} \underset{r\rightarrow 0}{\longrightarrow} 0.$$
The uniform Cauchy criterion ensures then that  $\mathcal R_v(x)$ converges, uniformly for $v\in [v_0,v_1] $,  to a real  number  $\tau(v)$ as $x$ goes to $x^{\star}$. Therefore,
\begin{equation}
\label{def:tau}\tau(v)= \lim_{x\to x^\star} \int_{x_M}^v\frac1u \left(\frac{1}{F_M(\phi_R(x, \tau_M(x,u)), u))}-\frac{1}{r_{\star}}\rr)du\end{equation}
and   the previous estimates also guarantee that $\tau$ is continuous on $(0,u_1]$ and that $\tau(v)$ goes to
$0$ as $v$ goes to $0$.
This ends the proof for  $v \in (0,u_1]$. \\

To deal with~$v\in [u_1, v_{\star})$, we split the time to reach~$v$ using the time to reach first~$u_1$:
$$\tau_M(x,v)=\tau_M(x,u_1)+\tau_M(x,u_1,v),$$
where $\tau_M(x,u_1,v)$ yields the time to go from $u_1$ to $v$:
\begin{equation}
    \label{splittimes}
\tau_M(x,u_1,v)=\tau_M((\phi_R(x,\tau_M(x,u_1)),u_1),v)=\int_{u_1}^v\frac1u \frac{1}{F_M(\phi_R(x, \tau_M(x,u)), u))}\, du.
\end{equation}
It remains to check that this  term converges as 
$x$ goes to $x^*$  (uniformly for
$u\in [u_1,u_3]$ where $u_3<v_{\star})$.  This can be achieved as above using uniform continuity of
$(x,u) \in \{(x,u) :  x\in B_{u_1}(x^\star), u\in [u_1,u_3]\} \rightarrow F_M(\phi_R(x, \tau_M(x,u)), u)$. 
It ensures that the limit is continuous and increasing with respect to $v$. {The conclusion follows from the fact that $\log(v/x_M) = \log(v/u_1) + \log(u_1/x_M)$.}\\
The case $\partial_R F_R(x^{\star})=0$ is treated in similar way.  The linearization argument for proving uniform continuity is more subtle in that partially hyperbolic case and we can use \cite{Takens}.
  \end{proof}

\subsection{Quantitative approximation of the stochastic process}

\label{qttdyn}

The main result  can now be stated. We consider a large number of mutants but authorize this number to be small compared to initial resident population size. This makes the limiting process to be null on finite time intervals and we look at large time intervals $\tau_M(x_K,v)$ and renormalized values to describe  how the mutant density escapes from the absorbing and unstable boundary. Thus the initial density of mutants is small, and the initial density of residents is close to the equilibrium.
  \begin{theorem}\label{maindynamic}
Under Assumptions~\ref{assumpt} and \ref{hypdom},  for all  $x_K=(x_R^K,x_M^K)\in { (\n^*/K)^2}$,
satisfying
$$1/K\ll x^K_M \ll 1 \quad  \textrm{ and  } \quad |x^K_R - x^{\star}_R| \log\left(1/x^K_M\rr)\ll 1,$$
 there exists $\eta_0>0$ such that, for any $v\in (0,v_\star),\eta\in(0,\eta_0)$, there exists some $C>0$ satisfying for all $K\in\mathbb{N}^*$,
$$\pro{\underset{t \leq \tau_M(x_K,v)}{\sup}\bigg\vert \frac{X^K_M(x_K,t)}{ \phi_M(x_K,t)}-1\bigg\vert>\eta}\leq \frac{C}{\sqrt{K x^K_M}}.$$
\end{theorem}
Observe that our initial conditions are satisfied
as soon as the initial number of mutants $N^K_M(0)$
goes to infinity but remains negligible compared to $K$ and the initial density of mutants $X^K_R(0)$ is close
to the equilibrium value up to a term of order of magnitude $1/\log K$.
 Our result stops at hitting times $\tau_M(x_K,v)$  but  could be easily extended on finite time intervals and with initial conditions of order one, following classical results of \cite{K1,K2}. \\

In all this subsection, the hypotheses of Theorem~\ref{maindynamic} are in force, and we denote, for $\bullet\in\{R,M\},$
$$X^K_\bullet(t) := X^K_\bullet(x_K,t)\,,\,
x^K_\bullet(t) := \phi_\bullet(x_K,t)$$
and
$$
Y^K_\bullet(t) := X^K_\bullet(t) - x^K_\bullet(t).$$

In addition, we introduce the following stopping times: for $\bullet\in\{R,M\},$
\begin{align*}
\theta^K_\bullet :=& \inf\left\{t>0~:~\frac{X^K_\bullet(t)}{x^K_\bullet(t)}\geq 2\rr\}\,,\,\theta^K := \theta^K_R\wedge\theta^K_M,\\
\sigma^K_\eta :=& \inf\left\{t>0~:~\int_0^t \underset{r\leq s}{\sup}~|Y^K_R(r)|ds\geq \eta\right\}\textrm{ for  }\eta>0,\\
\tilde\sigma^K_\eta :=& \inf\left\{t>0~:~\underset{s\leq t}{\sup}~|Y^K_R(s)|\geq \eta\right\}\textrm{ for }\eta>0.
\end{align*}

In order to prove Theorem~\ref{maindynamic}, we need the following result providing a control in probability when $K$ goes to infinity of the stopping times~$\theta^K_R$,~$\sigma^K_\eta$ and~$\tilde\sigma^K_\eta.$ {This lemma guarantees the same property as Lemma~\ref{controlRbranch} for the previous section: with probability going to one as $K$ goes to infinity, $\tau_M(x_K,L\eta)\wedge\theta^K_M$ is non-greater than the stopping times defined above. In other words, the stopping condition is reached by the mutant population and not by the resident one.}

\begin{lemme}\label{controltps}
There exist $L,\eta_0>0$ such that for all~$0<\eta\leq\eta_0$,
\begin{align}
    \pro{\sigma^K_\eta \leq \tau_M(x_K,L\eta)\wedge\theta^K}\leq& \frac{C}{\eta}K^{-1/2}\left(\log\left(1/x^K_M\right)\right)^{{3/2}},\label{controlintegraly}\\
        \pro{\tilde\sigma^K_{\eta} \leq \tau_M(x_K,L\eta)\wedge\theta^K\wedge\sigma^K_\eta}\leq& \frac{C}{\eta} K^{-1/2} \left(\log\left(1/x^K_M\right)\right)^{1/2},\label{controlsupy}\\
    \pro{\theta^K_R\leq \tau_M(x_K,L\eta)\wedge\theta^K_M}\leq& \frac{C}{\eta} K^{-1/2}\left(\log\left(1/x^K_M\right)\right)^{{3/2}}.\label{thetaKR}
\end{align}
\end{lemme}

\begin{proof}
To begin with, let us prove~\eqref{controlintegraly}. For the sake of notation, let us denote
$$S^K := \tau_M(x_K,L\eta)\wedge\theta^K\wedge\sigma^K_\eta.$$
By definition of~$\sigma^K_\eta,$
$$\left\{\sigma^K_\eta \leq \tau_M(x_K,L\eta)\wedge\theta^K\right\}= \left\{\int_0^{S^K} \underset{r\leq s}{\sup}\left|Y^K_R(r) \rr|ds \geq\eta\right\}.$$
To control the probability of this event, 
we use that
\begin{equation}
\label{composte}
Y^K_R(t) = H^K(t) + \int_0^t Y^K_R(s)\partial_R G_R(x^K(s))ds
\end{equation}
 for all $t>0,$ where
\begin{align*}
\displaystyle H^K(t) &:= M^K(t) + R^K(t) + \int_0^t Y^K_M(s)\partial_M G_R(x^K(s))ds,\\
\displaystyle M^K(t) &:= \frac1K \int_{[0,t]\times\r_+} \uno{z\leq K X^K_R(s-)b_R(X^K(s-))}\widetilde\N_R^b(ds,dz) \\
&\qquad \qquad  - \frac1K \int_{[0,t]\times\r_+} \uno{z\leq K X^K_R(s-)d_R(X^K(s-))}\widetilde\N_R^d(ds,dz),\\
\displaystyle R^K(t) &:= \int_0^t \left[G_R(X^K(s)) - G_R(x^K(s)) - Y^K_M(s)\partial_M G_R(x^K(s)) - Y^K_R(s)\partial_R G_R(x^K(s))\rr]dr.
\end{align*}

In a first time, let us treat the case $\partial_R G_R(x^{\star}) = x^{\star}_R \partial_R F_R(x^{\star})<0.$ Thanks to Lemma~\ref{lemmeadrien} we have
$$\underset{s\leq t}{\sup}\left|Y^K_R(s)\rr|\leq C\underset{s\leq t}{\sup}~\left|H^K(s) - H^K(\lfloor s\rfloor)\rr| \leq C~\underset{s\leq t}{\sup} |H^K(t)|.$$

On the other hand, if $\partial_R G_R(x^{\star}) = x^{\star}_R \partial_R F_R(x^{\star})=0,$ 
$$\int_0^{\tau_M(x_K,L\eta)}
|\partial_R G_R(x^K(r))|\leq C
\int_0^{\tau_M(x_K,L\eta)} \left( x^K_M(r) + |x^K_R(r) - x^{\star}_R| \right)dr\leq C L\eta,$$
using respectively Assumption \ref{assumpt} (R) and Lemma~\ref{xKRx*Rdeux}.
Besides, by Gronwall's lemma,
$$|Y^K_R(t)|\leq \sup_{s\leq t} |H^K(s)| \exp\left(
\int_0^t \vert \partial_R G_R(x^K(r))\vert dr \right).$$
Combining these two estimates implies that, for all $t\leq \tau_M(x_K,L\eta)\wedge\theta^K,$
$$|Y^K_R(t)| \leq \sup_{s\leq t} |H^K(s)| e^{C L\eta}\leq C \sup_{s\leq t} |H^K(s)|,$$
for $\eta\leq \eta_0$ fixed to some small enough value.

Gathering the two cases, we have proved that
if $\partial_R F_R(x^{\star})\leq 0$, for $t\leq \tau_M(x_K,L\eta)\wedge\theta^K,$
\begin{align}
\underset{s\leq t}{\sup}\left|Y^K_R(s)\rr| \leq& C_0~\underset{s\leq t}{\sup} |H^K(s)|\nonumber\\
\leq& C_0~\underset{s\leq t}{\sup} |M^K(s)|+C_0~\underset{s\leq t}{\sup} |R^K(s)|+C_0\int_0^t|Y^K_M(s)|\cdot|\partial_M G_R(x^K(s))|ds,\label{controlHK}
\end{align}
for some constant~$C_0>0$ whose value is fixed.

In addition, for $t\leq \tau_M(x_K,L\eta),$
\begin{align}
&\int_0^t x^K_M(s)ds = \int_0^t \frac{1}{F_M(x^K(s))} (x^K_M)'(s)ds \nonumber\\
&\qquad \qquad \leq C \int_0^t (x^K_M)'(s)ds = C(x^K_M(t) - x^K_M(0)) \leq C x^K_M(t).
\label{xkm2}
\end{align}
Using twice this inequality and
$|Y^K_M(r)|=x^K_M(r)\vert X^K_M(r)/x^K_M(r)-1\vert \leq x^K_M(r)$ for $r\leq \theta^K$,
\begin{align}
    \int_0^{S^K}\int_0^s|Y^K_M(r)|\cdot |\partial_M G_R(x^K(r))|drds
   & \leq C\int_0^{S^K}\int_0^s x^K_M(r)drds  
\leq C x^K_M(S^K)    \leq CL\eta.\label{intykm}
\end{align}
By Taylor-Lagrange's inequality and Lemma~\ref{xKRx*R} (which guarantees that the process $x^K$ is bounded up to time~$S^K$), we also have
\begin{align*}
    \int_0^{S^K}\underset{r\leq s}{\sup}~|R^K(r)|ds\leq C\int_0^{S^K} \int_0^s \left(|Y^K_M(r)|^2+ |Y^K_M(r)|\cdot|Y^K_R(r)| +|Y^K_R(r)|^2 \right) drds.
\end{align*}
Since $S^K\leq \tau_M(x_K,L\eta),$ and using the same computation as~\eqref{intykm},
$$\int_0^{S^K}\int_0^s |Y^K_M(r)|^2drds\leq L\eta\int_0^{S^K}\int_0^s |Y^K_M(r)|drds \leq C(L\eta)^2.$$
Besides $|Y^K_M(r)|\cdot|Y^K_R(r)|\leq |Y^K_M(r)|\underset{r'\leq s}{\sup}|Y^K_R(r')| $ for $r\leq s$,
\begin{align*}
\int_0^{S^K}\int_0^s |Y^K_M(r)|\cdot|Y^K_R(r)|drds\leq 
 CL\eta\int_0^{S^K}\underset{r'\leq s}{\sup}|Y^K_R(r')|ds
\leq CL\eta^2.
\end{align*}
Similarly,
$\int_0^{S^K}|Y^K_R(r)|^2drds 
\leq C\eta\int_0^{S^K}\underset{r'\leq s}{\sup}|Y^K_R(r')|ds\leq C\eta^2.$
Combining these estimates yields
 \begin{align}
 \label{encore}
 \int_0^{S^K}\underset{r\leq s}{\sup}~|R^K(r)|ds\leq
  C\eta^2.
\end{align}
Plugging~\eqref{intykm} and~\eqref{encore} in~\eqref{controlHK},
we can choose $L>0$ and $\eta>0$ small enough and get
\begin{align*}
\pro{\int_0^{\tau_M(x_K,L\eta)\wedge \theta^K}|Y^K_R(s)|ds>\eta}\leq& \pro{\int_0^{\tau_M(x_K,L\eta)\wedge \theta^K}|M^K(s)|ds>\frac{\eta}{3C_0}}.
\end{align*}
Using again Doob and Cauchy-Schwarz inequalities,
\begin{align*}
\esp{\underset{t\leq \tau_M(x_K,L\eta) \wedge \theta^K}{\sup}\left| M^K(t)\rr|}
&\leq  C
\esp{\langle M^K\rangle_{\tau_M(x_K,L\eta) \wedge \theta^K}}^{1/2}\\
&\leq C K^{-1/2}\esp{\int_0^{\tau_M(x_K,L\eta) \wedge \theta^K} X^K_R(s)(b_R+d_R)(X^K(s))ds}^{1/2}\\
&\leq C K^{-1/2}\tau_M(x_K,L\eta)^{1/2}.
\end{align*}
Using  Markov's inequality and the comparison of $\tau_M(x_K,L\eta)$ with  $\log\left(1/x^K_M\right)$
(cf Lemma~\ref{xKRx*R}$.iii)$),  we get
$$\pro{\int_0^{\tau_M(x_K,L\eta)\wedge \theta^K}|Y^K_R(s)|ds>\eta}\leq \frac{C}{\eta} K^{-1/2}\left(\log\left(1/x^K_M\right)\right)^{{3/2}}.$$
This proves~\eqref{controlintegraly}.\\

Now we prove~\eqref{controlsupy}. By definition of~$\tilde\sigma^K_{\eta},$
$$\Big\{\tilde\sigma^K_{\eta} \leq S^K\Big\} = \left\{\underset{s\leq S^K\wedge \tilde \sigma^K_{\eta}}{\sup}|Y^K_R(s)|> \eta\right\}.$$

For the sake of readability, let us denote
$$T^K := S^K\wedge\tilde\sigma^K_{\eta}.$$
To control the probability of the event above, we use~\eqref{controlHK} again. We  have that
\begin{align*}
&\int_0^{T^K}|Y^K_M(s)|\cdot|\partial_M G_R(x^K(s))|ds\\
&\qquad\leq C \int_0^{T^K}x^K_M(s)\frac{|Y^K_M(s)|}{x^K_M(s)} ds\leq C \int_0^{\tau_M(x_K,L\eta)}x^K_M(s)ds\\
&\qquad \leq C\int_0^{\tau_M(x_K,L\eta)}\frac1{F_M(x^K(s))} (x^K_M)'(s)ds\leq C(x^K_M(\tau_M(x_K,L\eta)) - x^K_M(0))\leq CL\eta,
\end{align*}
and
\begin{align*}
    \underset{s\leq T^K}{\sup}|R^K(s)|\leq& C\int_0^{T^K}|X^K_M(s) - x^K_M(s)|^2ds+C\int_0^{T^K}|Y^K_R(s)|^2ds+C\int_0^{T^K}|Y^K_R(s)|\cdot|Y^K_M(s)|ds\\
    \leq& C(L\eta)^2 + C L\eta^2 + C(L\eta)^{2}\leq C\eta^{2}.
\end{align*}
Since~$\eta$ and~$L$ can be chosen arbitrary small, we have
\begin{align*}
\pro{\underset{s\leq T^K}{\sup}|Y^K_R(s)|>\eta}\leq& \pro{\underset{s\leq T^K}{\sup}\left|M^K_{s}\right|>\eta/4}\\
\leq& C \eta^{-1} \esp{\langle M^K\rangle_{T^K}}^{1/2}
\leq C\eta^{-1} K^{-1/2} \left(\log\left(1/x^K_M\right)\right)^{1/2},
\end{align*}
which proves~\eqref{controlsupy}.

Finally, we prove~\eqref{thetaKR} using~\eqref{controlintegraly} and~\eqref{controlsupy}. Notice that
    $$\Big\{\theta^K_R\leq \tau_M(x_K,L\eta)\wedge\theta^K_M\Big\}\subseteq \left\{\underset{t\leq \tau_M(x_K,L\eta)\wedge\theta^K}{\sup}\frac{|Y^K_R(t)|}{x^K_r(t)}\geq 1\right\}.$$
Then, by Lemma~\ref{xKRx*R}, we know that for all $t\leq \tau_M(x_K,L\eta),$ if $\eta>0$ is small enough,
    $$x^K_R(t) \geq x^{\star}_R - C L\eta\geq x^{\star}_R/2.$$
     Hence, for $\eta>0$ small enough,
    \begin{align*}
        \Big\{\theta^K_R\leq \tau_M(x_K,L\eta)\wedge\theta^K_M\Big\}\subseteq&\left\{\underset{t\leq \tau_M(x_K,L\eta)\wedge\theta^K}{\sup}|Y^K_R(t)|\geq x^{\star}_R/2\right\}\\
        \subseteq&\Big\{\tilde\sigma^K_{L\eta} \leq \tau_M(x_K,L\eta)\wedge\theta^K\Big\}\\
        \subseteq&\Big\{\sigma^K_\eta \leq \tau_M(x_K,L\eta)\wedge\theta^K\Big\}\cup\Big\{\tilde\sigma^K_{L\eta} \leq \tau_M(x_K,L\eta)\wedge\theta^K\wedge\sigma^K_\eta\Big\}.
    \end{align*}
    Since the probabilities of the two last events above are controlled respectively by~\eqref{controlintegraly} and~\eqref{controlsupy}, the result is proved. 
\end{proof}


{We are now able to prove the main result of this section.}

\begin{proof}[Proof of Theorem~\ref{maindynamic}] We decompose
the proof of Theorem~\ref{maindynamic} in three steps. The first step is the approximation of $X^K_M$ by $x^K_M$ until the level of the mutant population dynamics reaches some  small macroscopic value. The second step is more classical and consists in showing that the approximation still holds true during an additional time length ~$T>0.$ In the last step, we
show that within these time intervals the mutant population dynamics reached levels $v_0<v^*$.\\


{\it Step~1.} In this first step, we prove that there exist ${C_0},C,L,\eta_0>0$ 
such that for any $\eta\in (0,\eta_0)$ and $K\geq 1$,
{\begin{equation}\label{maindyn1}
	\pro{\underset{t\leq \tau_M(x_K,L\eta)}{\sup}\bigg|\frac{X^K_M(t)}{ x^K_M(t)}-1\bigg|>C_0\eta}\leq C\cdot \frac1{\eta}\cdot \frac{1}{\sqrt{K x^K_M(0)}}.
\end{equation}
}

We use that
	\begin{align*}
	\frac{X^K_M(t)}{x^K_M(t)}
	=& 1 - \int_0^t \frac{X^K_M(s)}{x^K_M(s)}\, (b_M-d_M)(x^K(s))\,ds\\
	& + \int_{[0,t]\times\r_+} \frac1{K x^K_M(s)}\uno{u\leq K\cdot X^K_M(s-) \cdot b_M(X^K(s-))}\N^b_M(ds,du)\\
	& - \int_{[0,t]\times\r_+} \frac1{K x^K_M(s)}\uno{u\leq K\cdot X^K_M(s-) \cdot d_M(X^K(s-))}\N^d_M(ds,du).
	\end{align*}
Writing $\widetilde\N^b_M$ and $\widetilde\N^d_M$ the compensated measure of $\N^b_M$ and $\N^d_M,$ and introducing the locally square integrable martingale
	\begin{align*}
	E_M^K(t) := &\int_{[0,t]\times\r_+} \frac1{K x^K_M(s)}\uno{u\leq K\cdot X^K_M(s-) \cdot b_M(X^K(s-))}\widetilde\N^b_M(ds,du)\\
	& - \int_{[0,t]\times\r_+} \frac1{K x^K_M(s)}\uno{u\leq K\cdot X^K_M(s-) \cdot d_M(X^K(s-))}\widetilde\N^d_M(ds,du),
	\end{align*}
	we have
	$$\frac{X^K_M(t)}{x^K_M(t)}= 1 + E^K_M(t)+ \int_0^t \frac{X^K_M(s)}{x^K_M(s)}\left(F_M(X^K(s)) - F_M(x^K(s))\rr)ds.$$
	
 For convenience, let us denote  $$T^K=\tau_M(x_K,L\eta)\wedge \theta^K.$$  We observe that 
	\begin{equation*}
	\langle E^K_M\rangle_{T^K} = \int_0^{T^K} \frac1{K x^K_M(s)}\cdot\frac{X^K_M(s)}{x^K_M(s)}\cdot(b_M+d_M)(X^K(s)) \, ds\leq \frac{2C}{K} \int_0^{T^K} \frac{1}{x^K_M(s)}ds.
	\end{equation*}
	
	Consequently, for $T\leq T^K$,
	\begin{align*}
	&\underset{t\leq {T}}{\sup} \left|\frac{X^K_M(t)}{x^K_M(t)}-1\rr|\\
	&\qquad \leq \underset{t\leq {T}}{\sup}\left|E^K_M(t)\rr|+ C\int_0^{T}\left(\left|X^K_M(s) - x^K_M(s)\rr| + \left|X^K_R(s) - x^K_R(s)\rr|\rr)ds\\
&\qquad \leq  \underset{t\leq {T}}{\sup}\left|E^K_M(t)\rr|+ C\int_0^{T}x^K_M(s)\cdot\underset{r\leq s}{\sup}\left|\frac{X^K_M(r)}{x^K_M(r)}-1\rr|ds + C\int_0^{T}\left|X^K_R(s) - x^K_R(s)\rr|ds.
	\end{align*}
 So, by Gronwall's lemma, for any $T\leq \tau_M(x_K,L\eta)\wedge\theta^K,$
	\begin{equation}\label{principal}
	\underset{t\leq {T}}{\sup} \left|\frac{X^K_M(t)}{x^K_M(t)}-1\rr|\leq \left(\underset{t\leq {T}}{\sup}|E^K_M(t)| + C\int_0^T\left|X^K_R(s) - x^K_R(s)\rr|ds\rr)e^{C\int_0^T x^K_M(s)ds}.
	\end{equation}
Besides, by Burkholder-Davis-Gundy's inequality, 
	\begin{equation*}
	\esp{\underset{t\leq {T^K}}{\sup}|E^K_M(t)|}\leq C\esp{\langle E^K_M\rangle_{T^K}^{1/2}}\leq C K^{-1/2} \esp{\int_0^{T^K}\frac{1}{x^K_M(s)}ds}^{1/2},
	\end{equation*}
    {where the second inequality can be obtained either by Cauchy-Schwarz' inequality or Jensen's inequality.}
 
	Besides  $F_M(x^K(t))$ is lower-bounded on $[0,\tau_M(x_K,L\eta)]$ by some $\kappa>0$ thanks to Lemma~\ref{xKRx*R} since $F_M(x^{\star})>0$ if $\eta$ is chosen small enough. Then $x^K_M(s)\geq e^{\kappa s}x^K_M(0)$ and 
	\begin{equation}\label{ekmbdg}
	\esp{\underset{t\leq T^K}{\sup}|E^K_M(t)|}\leq 
 C \cdot \frac{1}{\sqrt{\kappa}} \cdot \frac{1}{\sqrt{K x^K_M(0)}},
	\end{equation}
	where the value of $C$ has changed in the last inequality, but is still independent of~$K.$\\
Let us recall that, thanks to~\eqref{xkm2}, we have
    \begin{equation*}
    \int_0^{\tau_M(x_K,L\eta)} x^K_M(s)ds \leq CL\eta.
    \end{equation*}
So, \eqref{principal} can be rewritten as
	\begin{equation}\label{eqC1}
	\underset{t\leq \tau_M(x_K,L\eta)\wedge\theta^K}{\sup} \left|\frac{X^K_M(t)}{x^K_M(t)}-1\rr|\leq C_1 \underset{t\leq \tau_M(x_K,L\eta)\wedge\theta^K}{\sup}|E^K_M(t)| + C_1 \int_0^{\tau_M(x_K,L\eta)\wedge\theta^K} \left|Y^K_R(s)\rr|ds,
	\end{equation}
 {for some positive constant~$C_1>0$.}
By Markov's inequality and~\eqref{ekmbdg},
$$\pro{\underset{t\leq T^K}{\sup}|E^K_M(t)| >\eta}\leq \frac{C}{\eta\sqrt{K x^K_M(0)}}.$$

Besides, by~\eqref{controlintegraly} from Lemma~\ref{controltps},
\begin{align*}
&\pro{\int_0^{T^K} \left|Y^K_R(s) \rr|ds >\eta}\leq \pro{\sigma^K_{\eta} \leq T^K} \\
&\qquad \qquad \qquad \qquad  \leq \frac{C}{\eta}K^{-1/2}\left(\log\left(\frac{1}{x^K_M(0)}\right)\right)^{{3/2}}
{\leq \frac{C}{\eta}K^{-1/2} \frac{1}{\left(x^K_M(0)\right)^{1/2}},}
\end{align*}
{where we have used that, for any~$x>0$ small enough, $\log (1/x)^3 \leq 1/x.$}

We obtain from the three last inequalities that
$$\pro{\underset{t\leq T^K}
{\sup} \bigg|\frac{X^K_M(t)}{x^K_M(t)}-1\bigg|> 2C_1\eta}\leq  \frac{C}{\eta}\cdot \frac{1}{\sqrt{K x^K_M(0)}}.$$
Using now that for $\eta\in (0,1)$, we obtain
\begin{align*}
\Big\{\theta^K_M<\infty, \, \theta^K_M \leq \tau_M(x_K,L\eta)\Big\}
\subseteq \left\{\underset{t\leq T^K}{\sup}\bigg|\frac{X^K_M(t)}{x^K_M(t)}-1\bigg|>\eta\right\},
\end{align*}
and $\eqref{thetaKR}$  proves~\eqref{maindyn1} with $C_0=2C_1$. \\

{\it Step~2.} We fix $T>0$ and   we prove now that the approximation of $X^K_M$ by $x^K_M$ is still valid on the time 
interval $[\tau_M(x_K,L\eta), \tau_M(x_K,L\eta)+T]$. The fact that the mutant population dynamics has left the neighborhood of $0$ makes such estimates more classical.
More precisely, we introduce the notation 
$$\tilde X^K_\bullet(t) := X^K_\bullet(t + \tau_M(x_K,L\eta)), \qquad \tilde x^K_\bullet(t) := x^K_\bullet(t + \tau_M(x_K,L\eta))$$
for  $\bullet\in\{R,M\}$.
First, we observe that
$$\gamma_{T,\eta}=\underset{t\leq T, K\geq 1}{\inf} \tilde x^K_M(t) \, >0,$$
since the boundary $\R_+\times \{0\}$ is stable 
for the flow $\phi$, which is {continuously differentiable}, and the initial value $\tilde x^K(0)$ belongs to a compact set of $(\R_+^*)^2$ (thanks to Lemma~\ref{xKRx*R})).

Then
$$\pro{\underset{t\leq T}{\sup}\frac{\left|\tilde X^K_M(t)-\tilde x^K_M(t)\right|}{\tilde x^K_M(t)}>\eta}\leq \pro{\underset{t\leq T}{\sup}\left|\tilde X^K_M(t)-\tilde x^K_M(t)\right| \, > \, \eta \, \gamma_{T,\eta}},$$
and it remains to prove that for any $\eta>0$,
\begin{equation}\label{tildeXKxK}
\pro{\underset{t\leq T}{\sup}\left|\tilde X^K_M(t)-\tilde x^K_M(t)\right|>\eta}\leq C K^{-1/2}.
\end{equation}

%
%
%
This result is a classical approximation of population processes by dynamical systems on finite time intervals, see Ethier-Kurtz \cite{ethierkurtz}. For sake of completeness and to give explicit bounds, let us prove the result here.
We use again  a stopping time which ensures boundedness :
$$\tilde\theta^K := \inf\Big\{t>0~:~\tilde X^K_M(t) > 2 \tilde x^K_M(t)\textrm{ or }\tilde X^K_R(t) > 2 \tilde x^K_R(t)\Big\}.$$

We have, for any $t\leq \tilde\theta^K$,
$$\left|\tilde X^K_\bullet (t) - \tilde x^K_\bullet(t)\rr|\leq \left|M^K_\bullet(t)\rr| + C\int_0^t\left(|\tilde X^K_R(s) - \tilde x^K_R(s)| + |\tilde X^K_M(s) - \tilde x^K_M(s)|\rr)ds,$$
where $M^K_R,M^K_M$ are locally square integrable martingales satisfying
$$\langle M^K_\bullet\rangle_{t\wedge\tilde\theta^K} \leq C t K^{-1}.$$

Then, Gronwall's lemma implies that
$$\underset{t\leq T\wedge\tilde\theta^K}{\sup}\left(\left|\tilde X^K_M(t) - \tilde x^K_M(t)\rr| + \left|\tilde X^K_R(t) - \tilde x^K_R(t)\rr|\rr)\leq \left(\underset{t\leq T\wedge\tilde\theta^K}{\sup}|M^K_R(t)| + \underset{t\leq T\wedge\tilde\theta^K}{\sup}|M^K_M(t)|\rr)e^{CT}.$$

Hence, using Markov's and Burkholder-Davis-Gundy's inequalities, we obtain
\begin{equation}\label{tildex}
\pro{\underset{t\leq T\wedge\tilde\theta^K}{\sup}\left(\left|\tilde X^K_M(t) - \tilde x^K_M(t)\rr| + \left|\tilde X^K_R(t) - \tilde x^K_R(t)\rr|\rr)>\eta'}\leq \frac1{\eta'} C_{T} K^{-1/2}.
\end{equation}

Finally, let us now dismiss the stopping time~$\tilde\theta^K.$ Let $A^K_{\eta'}(T)$ be the event
$$A^K_{\eta'}(T) := \left\{\underset{t\leq T}{\sup}\left(\left|\tilde X^K_M(t) - \tilde x^K_M(t)\rr| + \left|\tilde X^K_R(t) - \tilde x^K_R(t)\rr|\rr)>\eta'\rr\}.$$

{
Choosing~$\eta'$ smaller than
$$\underset{t\leq T,K\geq 1}{\inf}\tilde x^K_R(t) \wedge\underset{t\leq T,K\geq 1}{\inf}\tilde x^M_R(t),$$
guarantees that
$$\left\{\tilde \theta^K\geq T\right\}\subseteq A^K_{\eta'}(T\wedge\tilde\theta^K),$$
whence
\begin{align*}
\pro{A^K_{\eta'}(T)} =& \pro{A^K_{\eta'}(T) \cap\left\{\tilde\theta^K\geq T\right\}} + \pro{A^K_{\eta'}(T) \cap\left\{\tilde\theta^K< T\right\}}\\
\leq& \pro{A^K_{\eta'} (T\wedge\tilde\theta^K)} + \pro{A^K_{\eta'}(T)\cap A^K_{\eta'}(T\wedge\tilde\theta^K)}\\
\leq& 2\pro{A^K_{\eta'} (T\wedge\tilde\theta^K)}.
\end{align*}
}

Since~\eqref{tildex} provides a control of $\pro{A^K_{\eta'}(T\wedge\tilde\theta^K)},$ \eqref{tildeXKxK} is proved.\\

{\it Step~3.} We can now conclude the proof of the theorem
using the two steps. Combining the two results
yields the expect control until time
$\tau_M(x_K,L\eta)+T$, for any $\eta\in (0,\eta_0)$
and $T>0$. We just need to check that 
$T_v$ can be chosen (independently of $K$ and $x_K$)
so that $\tau_M(x_K,L\eta)+T_v\geq \tau_M(x_K,v)$. This is indeed the case, as can be seen from
\eqref{splittimes} with $u_1=L\eta$ for instance.
\end{proof}

\section{Thresholds hitting times}\label{hittingtimes}

In this section, we use the results of the two previous sections to approximate the hitting times of the process $(N^K_M(t))_t$. Let us recall that, for any  $n\geq 1$, we have defined in~\eqref{temps-mut} the stopping time
$$T^K_M(n_0, n) := \inf\left\{t>0~:~N^K_M(t)\geq n\right\}.$$
\\
The following main result of this section complements the results of \cite{BR, BHKK}.  It gives a quantitative information on the law of hitting times of thresholds of order $K$ but also for thresholds of size less than $K$. We recall that~$v_\star$ is defined at the beginning of Section~\ref{defvstar}.
\begin{theorem}\label{approxtps}
Grant Assumptions~\ref{assumpt},~\ref{assumptCI} and~\ref{hypdom}. 
Let $(\zeta_K)_K$ be some sequence which tends to infinity and satisfies $\zeta_K/K\rightarrow v$ as $K\rightarrow \infty$, for some $v \in [0,v_{\star})$.
Then,  for any $\varepsilon>0$, 
$$\lim_{K\rightarrow \infty}\P\left( \bigg\vert  T^K_{M}(N^K(0), \zeta_K) - \frac{\log(\zeta_K/W)}{r_{\star}}-  \tau(v) \bigg\vert \geq \varepsilon\,;\, W>0 \right)=0,$$
 where $\tau$ is the increasing continuous function such that $\tau(0)=0$ defined in 
 \eqref{def:tau} and
$$W=\lim_{t\rightarrow\infty}e^{-r_{\star}t}Z(t)\stackrel{d}{=}\frac{d_{\star}}{b_{\star}} \delta_0+ \frac{r_{\star}}{b_{\star}}\, \mu_{ r_{\star}/b_{\star}},\quad \mu_a(dx)=1_{R_+}(x)ae^{-ax}dx.$$
 \end{theorem}
 
 \medskip To prove this result, we first focus on the submacroscopic phase and use the branching process approximation to determine the asymptotic behavior of the  hitting times of the mutant population process. For this purpose, we need the following lemma which guarantees the hitting time $T^K_M(N^K(0),\xi_K)$ to go to infinity, for suitable sequence $(\xi_K)$.
 \begin{lemme}\label{tpsinfini}
 Under Assumptions~\ref{assumpt} and~\ref{assumptCI}, for any sequence~$\xi$ tending to infinity and  satisfying   $\xi_K\ll  K/\log K$  if $\partial_R F_R(x^{\star})<0$  (resp.  $ \xi_K\ll  K/(\log K)^2$ if $\partial_R F_R(x^{\star})=0$), the hitting time $T^K_M(N^K(0),\xi_K)$ goes to infinity almost surely as $K$ goes to infinity.
\end{lemme}

\begin{proof}
For the sake of notation, let us write $T^K_M := T^K_M(N^K(0),\xi_K)$ in all this proof.

By Lemma~\ref{controlRbranch}, for any $\eps>0$, there exists some $K_\eps\in\n^*$ such that, for all $K\geq K_\eps$,
$$\pro{\forall t\leq T^K_M,~b_M(X^K(t))\leq b_\star + 1}\geq 1-\eps.$$
In particular, we can couple the process $N^K_M$ with a branching process $Z_+$ starting at $Z_+(0)=1$ with individual birth rate~$b_+:=b_\star + 1$ and without death, such that, for all $K\geq K_\eps,$
$$\pro{\forall t\leq T^K_M,~N^K_M(t)\leq Z_+(t)}\geq 1-\eps.$$
It is known that $Z_+(t)e^{-b_+ t}$ is a local martingale converging to some finite r.v.~$W_+$ that is finite a.s. {(see e.g. Theorem~1 of \cite[p. 111]{AN} and Theorem~2 of \cite[p. 112]{AN} in Section~III.7)}. As a consequence $\bar W_+ := \underset{t\geq 0}{\sup} Z_+(t)e^{-b_+t}$ is also finite a.s. Introducing $T_Z^K$ the hitting time of the value~$\xi_K$ by the process~$Z_+$, we have
$$\left\{N^K_M(T^K_M)\leq Z_+(T^K_M)\right\}\subseteq\left\{T^K_M \geq T^K_Z\right\}.$$

Besides, by definition of $T^K_Z,$ on the event $\{T^K_Z<\infty\}$,
$$\lceil \xi_K\rceil = Z(T^K_Z) \leq \bar W_+ e^{b_+ T^K_Z},$$
which implies
$$T^K_Z \geq \frac{1}{b_+}\log(\lceil\xi_K\rceil/\bar W_+).$$

Consequently, for all~$\eps>0$,
$$\pro{T^K_M\underset{K\rightarrow\infty}{\longrightarrow} +\infty}\geq 1-\eps.$$
Letting $\eps$ go to zero proves the result.
\end{proof}

{The previous lemma allows us to obtain an asymptotic expansion for the hitting time~$T^K_M(N^K(0),\xi_K)$. One can note that the statement of Proposition~\ref{tpsbranch2} is stronger than the one of Lemma~\ref{tpsinfini}.}
 
 \begin{prop}\label{tpsbranch2}
 Under Assumptions~\ref{assumpt} and~\ref{assumptCI},  for any $\varepsilon>0$,
$$\lim_{K\rightarrow \infty} \P\left( \left| T^K_{M}(N^K(0), \xi_K) - \frac{\log(\xi_K/W)}{r_\star}\right| \geq \varepsilon \,  ; \, W>0  \right)=0$$ for any sequence~$\xi$ tending to infinity and  satisfying   $\xi_K\ll  K/\log K$ if $\partial_R F_R(x^{\star})<0$ (resp.  $ \xi_K\ll  K/(\log K)^2$  if $\partial_R F_R(x^{\star})=0$).
\end{prop}
\begin{proof} Consider the case $\partial_R F_R(x^{\star})<0$ and let $(\xi_K)_K$ be a sequence tending to infinity  and  satisfying   $\xi_K\ll  K/\log K$. We apply Corollary~\ref{mainbranch1}.
Then for  $\varepsilon>0$, there exists $K_0$ such that for any $K\geq K_0$,
$$ \mathbb{P}\bigg(\forall t\in [0, T^K_M], \,   (1-\varepsilon) Z(t)\leq N^K_M(t) \leq  (1+\varepsilon) Z(t) \, \vert  \, \ W>0\bigg)\geq 1-\varepsilon,$$
where we write $T^K_M=T^K_{M}(N^K(0),\xi_K)$ for convenience.
Since $Z(t)\exp(-r_{\star}t)$ converges to $W$ almost surely  as $t$ tends to infinity, 
$$\P\left(\sup_{t\geq u} \left| \frac{Z(t)}{\exp(r_{\star}t)W} - 1\right| \geq \varepsilon \, \vert  \, W>0 \right)\rightarrow 0$$
as $u$ goes to infinity. Then  there exists $t_0$ such that
$$\P\bigg( \forall t\geq t_0, \, (1-\varepsilon) W e^{r_{\star}t} \leq Z(t)\leq (1+\varepsilon) We^{r_{\star}t}    \, \vert  \, W>0\bigg)\geq 1-\varepsilon.$$

In addition, by Lemma~\ref{tpsinfini}, we know that $\P(T^K_M\geq t_0 \, \vert \, W>0)\geq 1-\varepsilon$ for $K$ large enough.
Combining these estimates,
$$ \mathbb{P}\bigg(\forall t\in [t_0, T^K_M], \,   (1-\varepsilon)^2 W e^{r_{\star}t} \leq N^K_M(t) \leq  (1+\varepsilon)^2W e^{r_{\star}t} \,   \vert  \, W>0\bigg)\geq 1-3\varepsilon.$$
Adding that 
$\xi_K \leq N^K_M(T^K_M) \leq  \xi_K+1$ by definition of $T^K_M$,  we obtain for $t=T^K_M$
$$ \mathbb{P}\bigg(\ \,   (1-\varepsilon)^2 W e^{r_{\star}T^K_M} \leq \xi_K +1, \,  \xi_K \leq  (1+\varepsilon)^2W e^{r_{\star}t} \,   \vert  \, W>0\bigg)\geq 1-3\varepsilon.$$
Finally, taking the $\log$ inside the last probability  yields the result.

{The case where $\partial_R F_R(x^\star_R,0)=0$ can be treated in a similar way, and is therefore omitted.}
\end{proof}

\medskip {Now, in order to prove Theorem~\ref{approxtps} from the results of the previous sections, we need to guarantee that, when two processes are close in some sense, their hitting times are also close. This property is stated and proved in the two next lemmas.}

\begin{lemme}\label{titilemme} Let $f,g : [0,t_0]\rightarrow \R$. 
 Assume that $g$ is differentiable and  
$$\alpha = \inf_{[0,t_0]} g' >0, \quad \beta=\sup_{t\in [0,t_0]} \vert f(t)-g(t)\vert.$$
{If $\beta<g(t_0)-g(0)$, then} for any $x\in [g(0),g(t_0)-\beta]$,  $\{t \in [0,t_0] :  f(t)\geq x\}$ is non empty and
$$\big\vert \inf\{t \in [0,t_0] :  f(t)\geq x\} - \inf\{ t\in [0,t_0] : g(t)\geq x\}\big\vert\leq \frac{\beta}{\alpha}.$$
\end{lemme}
\begin{proof}
Let us note in a first time that $\{t \in [0,t_0] :  f(t)\geq x\}$ is non empty since
$$f(t_0) = g(t_0) + (f(t_0) - g(t_0))\geq g(t_0) - \beta \geq x.$$

Let $t(x) \in [0,t_0]$ be the (unique) time when $g(t(x))=x$. Then for
$t\in [0 ,t(x)-\varepsilon]$, 
$$f(t)\leq g(t)+\beta \leq g(t(x))-\varepsilon \alpha +\beta =x -\varepsilon \alpha +\beta < x$$
where the last inequality holds for any $\varepsilon> \beta/\alpha$. 
This forces
 \ $\inf\{t \in [0,t_0] :  f(t)\geq x\}\geq t(x)-\varepsilon$.
Similarly for $t \in [ t(x)+\varepsilon, t_0]$ and $\varepsilon \geq \beta/\alpha$,
$$f(t)\geq g(t)-\beta \geq g(t(x)) + \varepsilon \alpha  -\beta \geq x.$$
This yields 
 $ \inf\{t \in [0,t_0] :  f(t)\geq x\} \leq t(x)+\varepsilon$ and completes the proof since $\varepsilon$ can be chosen arbitrarily close to
 $\beta/\alpha$.
\end{proof}

\medskip We can now capture the behavior of   hitting times for $X^K_M(t)$ when  $X^K_M(0)$ is  large compared  to $1/K$
using the approximation by the dynamical system.
For $x\in (\n/K)^2$ and $v\in \R_+$, we  consider the hitting times	
	\begin{equation}\label{SMK}
        \tau_M(x,v) = \inf\{t>0:x_M(t)\geq v \}, \quad \,  S_M^K(x,v) =
	\inf\{t>0:X^K_M(t)\geq v\},
 \end{equation}
	where  $x(0)=X^K(0)=x$.
\begin{lemme}\label{tpsdyn1} Under Assumptions~\ref{assumpt} and~\ref{hypdom},
assume 
$$ 1/K \ll X^K_M(0) \ll 1 \, \quad \text{ and } \, \quad |X^K_R(0)-x_R^{\star}|\log(1/X^K_M(0)) \ll 1$$ in probability as $K\rightarrow \infty$. 
For any $1<c_1$, $v_0<v_\star$ and $\eps>0$ small enough,
$$\lim_{K\rightarrow \infty} \P\left(\underset{c_1 X^K_M(0)\leq v\leq v_0}{\sup}\big|S^K_M(X^K(0),v)-\tau_M(X^K(0),v)\big|\geq \varepsilon\right)=0.$$
\end{lemme}

\begin{proof}
There exists $C>0$ such that for any $\varepsilon\in (0,1/2]$ and $u\in [1-\varepsilon, 1+\varepsilon]$,  $\vert \log(u) \vert\leq C \vert u-1\vert \leq C \varepsilon$. Using 
Theorem \ref{maindynamic}, we have
$$\lim_{K\rightarrow \infty} \P\left( \sup_{t\leq \tau_M(X^K(0),{v_0})} \bigg\vert \log\left(\frac{X^K_M(t)}{x_M^K(t)}\right)\bigg\vert  \geq C\varepsilon\right)=0.$$
We apply now Lemma \ref{titilemme} {on the event
$$\left\{\sup_{t\leq \tau_M(X^K(0),{v_0})} \bigg\vert \log\left(\frac{X^K_M(t)}{x_M^K(t)}\right)\bigg\vert  < C\varepsilon\right\}$$
}
with 
$$f(t)=\log(X^K_M(t)), \,  \quad g(t)=\log(x_M^K(t)),$$
{which is possible by choosing~$\eps$ small enough such that, for all $c_1 X^K_M(0)\leq v\leq v_0$ (with the notation of Lemma~\ref{titilemme})
$$\beta < C\eps \leq \log c_1 \leq \log(v/X^K_M(0)) = g(t_0) - g(0).$$}
By Lemma~\ref{lemmeetav}, there exist some $\eta>0$ and $K_0\in\n^*$ such that, for all $K\geq K_0$,
$$\pro{\forall t\leq\tau_M(X^K(0),{v_0}), F_M(x^K(t))\geq\eta}\geq 1-\eps.$$
Then, noticing that $g'(t)=F_M(x^K(t))$ and that $ \vert f(t) - g(t) \vert =\vert \log(X^K_M(t)/x_M^K(t))\vert $ allows to conclude.
\end{proof}
Finally, we end this section with the proof of the main result.
\begin{proof}[Proof of Theorem~\ref{approxtps}]
To begin with, if $\zeta_K\ll K/(\log K)^2$, then $v=0$ and the result has been proved in Proposition~\ref{tpsbranch2} above.

Otherwise, let $(\xi_K)_K$ be any sequence such that
$$\xi_K< \zeta_K\textrm{ and }\xi_K\ll K/(\log K)^2,$$
and let $v_K := \zeta_K/K$. Writing $\tilde X^K(t) := X^K(t+T^K_M(N^K(0),\xi_K))$, and using the notation of~\eqref{SMK},
$$T^K_M(N^K(0),\zeta_K) = T^K_M(N^K(0),\xi_K) + S^K_M(\tilde X^K(0), v_K).$$
Finally, since $\tilde X^K(0)$ satisfies the assumption of Lemma~\ref{tpsdyn1} (cf Lemma~\ref{controlRbranch}), the result is a direct consequence of Propositions~\ref{tKv},~\ref{tpsbranch2} and Lemma~\ref{tpsdyn1}.

{Note that, since $K/(\log K)^2\ll K/(\log K)$, we do not need to treat separately the cases where $\partial_R F_R(x^\star,0)$ is null or negative in the proof of Theorem~\ref{approxtps}.}
\end{proof}
In this part, we have obtained an approximation of hitting times which allows to quantify the stochastic effects. Developing more deeply our approach, we could study the speed of convergence in Theorems \ref{maindynamic} and \ref{approxtps}. This is left for a future work.

\section{Examples : competitive and epidemiological models}

\subsection{Lotka-Volterra competitive model}
\label{secECO}
We consider the dynamics of a two-dimensional population composed of resident and mutant individuals, which have their own demographic parameters and interact in a competitive way, for example in  sharing resources or niche areas. This model is included in  the framework of the paper, with constant individual birth  rates
$$b_\bullet(X^K(s-)) = b_\bullet $$ and individual death rates of the Lotka-Volterra form :
$$d_R(X^K(s))= d_R + c_{1,1}X^K_R(s) +c_{1,2} X^K_M(s),$$
$$d_M(X^K(s))= d_M + c_{2,1}X^K_R(s) + c_{2,2} X^K_M(s),$$
with  individual birth and death rates  $b_R, b_M\ge 0$ and $d_R, d_M\ge 0$ and a competition matrix $(c_{i,j})_{1\leq i,j\le 2}$  with non-negative coefficients. We assume that the  coefficients $b_R-d_R$, $b_M-d_M$, $c_{1,1}$ and $c_{2,2}$ are positive.

The corresponding dynamical system is the competitive Lotka-Volterra system, that is the ODE \eqref{def:dynsys}
with
$$F_{R}(x_{R},x_{M})=b_R-d_R-c_{1,1}x_R-c_{1,2} x_M, \qquad
F_{M}(x_{R},x_{M})=b_M-d_M-c_{2,1}x_R-c_{2,2} x_M.$$
Then
$$x^\star_R = \frac{b_R-d_R}{c_{1,1}}.$$
Let us check that the previous results can be applied.
First, we observe that the regularity in Assumption \ref{assumpt}
is satisfied.
Second,   $(x_R^\star,0)$
is an equilibrium and 
$$\partial_R F_R(x_R^\star,0)= - c_{1,1}<0.$$
This means that Assumption~\ref{assumpt}$.(E)$ is satisfied for $(x_R^\star,0)$. Moreover,   $x^\star_R$ is a stable equilibrium for the resident population alone, which corresponds to the first case considered in the branching process approximation.
\\
Finally, we start from a population with only one mutant, $N_R^K(0)=K-1, N_M^K(0)=1$ and we consider $x_R^*=1$. So  the invasion condition 
$F_M(x_R^*,0)>0$
reads $b_M-d_M-(b_R-d_R) c_{2,1}/c_{1,1}>0$. It is then easy to compute $v^\star$, which is
$$v^\star=x^\star_M=\frac{b_M-d_M}{c_{2,2}}
\,  \text{ if  } F_R(0,x^\star_M)<0, \quad  \text{ or } v^\star= \frac{(b_M-d_M)c_{1,1}-(b_R-d_R)c_{2,1}}{c_{1,1}c_{2,2}-c_{1,2}c_{2,1}} \, \text{ if  } \, F_R(0,x^\star_M)>0.$$

{It can also be noticed that Assumption~\ref{hypdom} is satisfied.}

In this framework, it has been proved in many papers, in particular in  \cite{durrett-mayberry}, \cite{Bovier2019}, \cite{champagnat} that new phenomena can emerge from small subpopulations whose size is of order $K^\alpha$, $\alpha < 1$. 

{Our results allow to characterize the size of populations on such scales. In particular, these results emphasize the fact that, on the scales $K^{\alpha}$ ($\alpha\in (0,1)$), two approximations are valid in the following sense (cf respectively to Theorems~\ref{mainbranch} and~\ref{maindynamic}):
\begin{prop}
For $c>0$,
\begin{itemize}
    \item[i)] starting from $N^K(0)=(K-1,1)$, for all $\eta>0$ and $\alpha\in(0,1)$,
    $$\pro{\underset{t\leq T^K_M(N^K(0),\lfloor cK^\alpha\rfloor)}{\sup}\left|\frac{N^K_M(t)}{Z(t)}-1\right| >\eta~;~\underset{t\geq 0}{\inf}~Z(t)\geq 1}\underset{K\rightarrow\infty}{\longrightarrow}0,$$
    where $Z$ is a birth and death process with individual birth rate $b_M$ and individual death rate $d_M+ c_{2,1}  x^\star_R$.
    \item[ii)] For $\alpha\in(0,1]$ and $\beta\in [0,1)$, assume that, almost surely, $N^K(0)$ satisfies
    $$1\ll N^K_M(0)\leq c K^\alpha\textrm{ and }|N^K_R(0) - K|\leq K^\beta.$$
    Then
    $$\pro{\underset{t\leq T^K_M(N^K(0),\lfloor cK^\alpha\rfloor)}{\sup}\left|\frac{N^K_M(t)}{Kx^K_M(t)}-1\right| >\eta}~\underset{K\rightarrow\infty}{\longrightarrow}0,$$
    where $(x^K_R(t),x^K_M(t))_t$ is the solution of
    $$\left\{\begin{array}{ll}
    \frac{d}{dt} x^K_R(t) &= x^K_R(t)(b_R - d_R - c_{1,1}x^K_R(t) - c_{1,2}x^K_M(t)),\\
    \frac{d}{dt} x^K_M(t) &=x^K_M(t) (b_M - d_M - c_{2,1} x^K_R(t) - c_{2,2} x^K_M(t)),
    \end{array}\right.$$
    with $(x^K_R(0),x^K_M(0)) = (N^K_R(0)/K,N^K_M(0)/K).$
\end{itemize}
\end{prop}
}

{In addition,} it can be essential to know at which time a population initiated by a single mutant individual could attain such size. Our results answers to the question for all $\alpha< 1$. More precisely,  Theorem \ref{approxtps} for $\zeta_K=cK^{\alpha}$
yields
\begin{corol} 
For any $\alpha \in (0,1)$, $c>0$ and $\eps>0$,
    $$\lim_{K\rightarrow \infty}\P\left( \bigg\vert  T^K_{M}(N^K(0), \lfloor cK^{\alpha}\rfloor) - \frac{\alpha\log(K)+\log(c/W)}{r_{\star}} \bigg\vert \geq \varepsilon\,;\, W>0 \right)=0,$$
\end{corol}
Here $v=0$ and $\tau(0)=0$.

\subsection{SIR model}
\label{secSIR}
Let $\beta>0$ be the infection rate (per pair of individuals) in a mixed
population and 
$\gamma>0$ the individual recovery rate.
Let $K\geq 1$ be the total population size.
At time $t$,  the number of susceptible individuals is denoted by $S^K(t)$,
the number of  infected individuals is denoted by  $I^K(t)$ and  $K-S^K(t)-I^K(t)$ yields the number 
of recovered individuals.
Each susceptible becomes infected at time 
$t$ with rate $\beta I^K(t)/K$.
Starting from one infected individual {and all other individuals susceptible}, the process $N^K=(N_R^K,N_M^K)= (S^K,I^K)$ is the unique strong solution of
\begin{align*}
S^K(t)=&K-1   -\int_{[0,t]\times\r_+}\uno{u\leq \beta S^K(s-)I^K(s-)/K}\N_I(ds,du).\\
I^K(t)=&1 +  \int_{[0,t]\times\r_+}\uno{u\leq \beta S^K(s-)I^K(s-)/K}\N_I(ds,du)-\int_{[0,t]\times\r_+}\uno{u\leq \gamma I^K(s-)}\N_G(ds,du),
\end{align*}
where $\N_I$ and $\N_G$ are Poisson point measures on $\R_+^2$ with intensity $dsdu$.\\

The process 
$N^K$
is a bitype birth and death process as considered in the introduction with $\N_I=\mathcal N_M^b=\mathcal N_R^d$ and  birth and death rates   defined by
\begin{align*}
b_R(x_R;x_M)&=0, \qquad \quad \, \, \,  d_R(x_R;x_M)=\beta  x_M\\
b_M(x_R;x_M)&=\beta  x_R, \qquad d_M(x_R;x_M)=\gamma.
\end{align*}
The ODE describing the limit of $N^K/K$  is given by \eqref{def:dynsys}
with
$$F_{R}(x_{R},x_{M})=-\beta x_M, \qquad F_{M}(x_{R},x_{M})=\beta x_R-\gamma.$$

Let us check that our results can be applied to the case where the first derivative of $F_R$ with respect to $x_R$ cancels.
First, we observe that regularity of Assumption \ref{assumpt}
is satisfied.
Second, for any $x_R>0$,   $(x_R,0)$
is an equilibrium and 
$$\partial_R F_R(x_R,0)=0.$$
This means that the second part of 
Assumption \ref{assumpt} is satisfied for any $x_R>0$. Moreover, there is no geometric stability for the resident population alone, which corresponds to the second case considered in the branching process approximation.\\
Finally, we start from a population with only one mutant, $S^K(0)=N_R^K(0)=K-1, I^K(0)=N_M^K(0)=1$ and we consider $x_R^*=1$. So  the invasion condition 
$F_M(x_R^*,0)>0$
reads $\beta>\gamma$.
The branching process $Z$ is a birth and death process  with individual birth rate $\beta$ and individual death rate $\gamma$.

In this example,  the maximal value of  the invasive population is the  peak of the epidemic and  can be computed. Indeed, at the peak time, the derivative of $x_I$ is zero, so  $x_S=\gamma/\beta$ Adding that $\gamma\log(x_S)+\beta (1-x_S-x_I)=0$, we obtain for the peak 
value $v_{\star}=x_I$
the expression: $$v_{\star}=1-\frac{\gamma}{\beta}+\frac{\gamma}{\beta}\log\left(\frac{\gamma}{\beta}\right).$$

We obtain the following results when there is an outbreak of the epidemic.  The hitting times for the infected population are defined by
$$\tau^K_{I}(n)=\inf\{t\geq 0,\,  I^K(t) \geq n \}.$$
\begin{prop} If $\beta>\gamma$, then\\
i)  for any $\xi_K \ll K/(\log K)^2$ and $\eta>0$,
\begin{align*}
\lim_{K\to +\infty} \mathbb{P}\bigg(\underset{t\leq \tau^K_{I}(\xi_K)}{\sup}~ \left|\frac{I^K(t)}{I(t)}-1\rr| > \eta \, ;  \, I(\tau^K_{I}(\xi_K ))>0\bigg)=0,
\end{align*} 
where $I$ is the unique strong solution of
\begin{align*}
I(t)=&1 + \int_{[0,t]\times\r_+}\uno{u\leq \beta\,I(s-) }\, \N_I(ds,du)
-\int_{[0,t]\times\r_+}\uno{u\leq  \gamma\, I(s-)}\, \N_G (ds,du).
\end{align*}
ii) For any $v<v_{\star}$ and $\varepsilon>0$, 
	$$\lim_{T\rightarrow \infty} \,
 \limsup_{K\rightarrow \infty} \, \P\left( \sup_{t \in [T,\tau^K_I(vK)]}\bigg\vert \frac{I^K(t)}{K x_I^{K,T}(t)}-1\bigg\vert
\,  \geq \varepsilon; \, I^K(T)>0\right)= \, 0,$$
where $(x^{K,T}_S(t),x^{K,T}_I(t))$ is the solution of 
\begin{eqnarray*}
\begin{cases}
\frac{d}{dt} x^{K,T}_S(t) &= -\beta x^{K,T}_S(t) x^{K,T}_I(t),\\
\frac{d}{dt} x^{K,T}_I(t) &= (\beta x^{K,T}_S(t)-\gamma) x^{K,T}_I(t) .
\end{cases}
\end{eqnarray*}
such that $(x^{K,T}_S(0),x^{K,T}_I(0))=(S^K(T)/K, I^K(T)/K)$.
\end{prop}
Point $i)$ is a consequence of Theorem \ref{mainbranch}
and
point ii) comes from Theorem \ref{maindynamic}.
These two parts have an intersection : in the time window when the number of infected individuals is large but negligible compared to $K/(\log K)^2$, we have informally 
$$I^K(t)\sim I(t)\sim W e^{(\beta-\gamma)t}\sim Kx_{I}^{K,T}(t-T).$$
We also obtain an approximation of the time when the epidemics reaches a given level (by Theorem~\ref{approxtps}):
\begin{corol} Let $(\zeta_K)_K$ be sequence which tends to infinity and satisfies $\zeta_K/K\rightarrow v$ as $K\rightarrow \infty$, for some $v \in [0,v_{\star})$.
Then,  for any $\varepsilon>0$, 
$$\lim_{K\rightarrow \infty}\P\left( \bigg\vert  \tau^K_{I}(\zeta_K) - \frac{\log(\zeta_K/W)}{r_{\star}}-  \tau(v) \bigg\vert \geq \varepsilon, \, W>0 \right)=0,$$
 where $\tau$ is the increasing continuous function such that $\tau(0)=0$ defined in Proposition \ref{tKv}.
\end{corol}
This question of how long it takes for an epidemic to reach a high level has several motivations. Indeed, sanitary capacities  for managing the infected individuals (hospital beds, isolation...) are limited.
The time at which sanitary capacities reach saturation is an important factor in managing epidemics and taking control measures (social distancing, confinement, etc.). In addition, we control here the period during which the stochasticity of the process still plays a role, before deterministic predictions are valid with sufficient accuracy.


\section{Appendix}\label{secappend}

\subsection{Minimum of the martingale associated with the branching process $Z$}\label{appendbranch}

Recalling the notation of Section~\ref{sectionbranch}, we are interested in the minimal value
of the martingale $W(t) := Z(t)\exp(-r_\star t)$. Such object for simple branching processes has already attracted a lot of attention. We are not aware in the literature of the estimates given here  and provide the proof for completeness. This relies on classical $L^2$ estimates to control the speed of the convergence of the martingale.

\begin{lemme}\label{propbranchement}
	 For any $\eps>0,$ and any $Z-$adapted stopping time $\tau,$ 
	$$\pro{W(\tau)\leq \eps\, ;\, W>0}\leq C\left(\eps +  \eps^{-1}\esp{\uno{\tau<\infty}e^{-r_\star \tau/2}}\right),$$
	where $W$ is the almost sure limit of $W(t)$ as $t$ goes to infinity, and $C$ is some positive constant independent of~$\eps$ and~$\tau.$ We use the convention $W(+\infty) := W$ for the inequality above to make sense.
\end{lemme}

\begin{proof}
	To simplify the proof, we firstly treat the case where $\tau$ is deterministic and finite. To begin with,
	$$W(\tau+1) - W(\tau) = e^{-r_\star \tau}\left(\sum_{i=1}^{Z(\tau)} N_i e^{-r_\star} -Z(\tau)\right),$$
	where, conditionally on $Z(\tau),$ the $N_i$ are i.i.d. with the same distribution as~$Z(1).$ This implies that
	\begin{align}
	\esp{\left(W(\tau+1) - W(\tau)\right)^2} =& e^{-2r_\star \tau}\esp{\mathbb{V}ar\left[\sum_{i=1}^{Z(\tau)} \left(N_i e^{-r_\star} -1\right)|Z(\tau)\right]} \nonumber\\
	=& e^{-2r_\star \tau} \esp{Z(\tau)} \mathbb{V}ar \left[N_i e^{-r_\star} -1\right]\leq C e^{-r_\star \tau}.\label{branchprop}
	\end{align}
	
	On the other hand, we write
	\begin{multline*}
	\pro{W(\tau)\leq \eps\, ;\, W>0}\leq \pro{W\leq 2\eps\, ;\, W>0} + \pro{W(\tau)\leq \eps\, ;\, W>2\eps}\\
	\leq \pro{W\leq 2\eps\, ; \, W>0} + \pro{W - W(\tau) > \eps}.
	\end{multline*}
	Recalling that, on the event $\{W>0\},$ $W$ follows an exponential distribution, we can bound the first term of the sum above by $1-e^{-C\eps} \leq C\eps.$
	
	The second term of the sum can be handled  as follows:
	\begin{align*}
	\pro{W - W(\tau) > \eps} =& \pro{\sum_{n\geq 0}W(\tau + n +1) - W(\tau + n)>\eps}\\
	\leq& \frac{1}{\eps}\sum_{n\geq 0}\esp{\left|W(\tau + n +1) - W(\tau + n)\right|}\\
	\leq &\eps^{-1}\sum_{n\geq 0}\esp{\left|W(\tau + n +1) - W(\tau + n)\right|^2}^{1/2}\\
	\leq &C\eps^{-1}e^{-r_\star \tau /2}\sum_{n\geq 0} e^{-r_\star n/2}\leq C \eps^{-1} e^{-r_\star\tau/2}.
	\end{align*}
	This ends the proof   in the case where $\tau$ is deterministic.
	
	Now to prove the result when $\tau$ is some almost surely finite $Z-$adapted stopping time, we just have to do the same computation conditionally on~$\mathcal{F}_\tau.$ Indeed, since $Z(\tau)$ is $\mathcal{F}_\tau-$measurable and $Z(\tau+1)-Z(\tau)$ independent of $\mathcal{F}(\tau)$, \eqref{branchprop} becomes
	$$\mathbb{E}\left[\left(W(\tau+1) - W(\tau)\right)^2 |\mathcal{F}_\tau\right] \leq Ce^{-r_\star \tau},$$
	and the last computation of the previous case gives
	$$\mathbb{P}\left(W-W(\tau)>\eps|\mathcal{F}_\tau\right) \leq C\eps^{-1} e^{-r_\star \tau/2}.$$
	The result of the lemma is then proved by taking the expectation in the inequality above.
	
	Finally let us treat the case where $\tau$ is not necessary almost surely finite.
	We write
	\begin{align}
	\pro{W(\tau)\leq\eps\, ; \,W>0} =& \pro{W(\tau)\leq\eps\, ; \,W>0\, ; \,\tau<\infty} + \pro{W(\tau)\leq\eps\, ; \,W>0\, ; \,\tau=\infty}\nonumber\\
	\leq& \pro{W(\tau)\leq\eps\, ; \,W>0\, ; \,\tau<\infty} + \pro{0<W\leq\eps}.\label{controlwtau}
	\end{align}
	
	Note that, on the event $\{\tau<\infty\},$ we have the almost sure convergence
	$$\uno{W(\tau\wedge n)\leq\eps;W>0} \underset{n\rightarrow\infty}{\longrightarrow}\uno{W(\tau)\leq \eps;W>0}.$$
	
	Hence, by Fatou's lemma,
	\begin{align*}
	\pro{W(\tau)\leq\eps\, ; \,W>0\, ; \,\tau<\infty} \leq& \underset{n}{\lim\inf}~\pro{W(\tau\wedge n)\leq\eps\, ; \,W>0\, ; \,\tau<\infty}\\
	\leq& \underset{n}{\lim\inf}~\pro{W(\tau\wedge n)\leq\eps\, ; \,W>0}. 
	\end{align*}
	
	Then, applying the result of the lemma with the almost surely finite stopping time $\tau\wedge n,$ we obtain
	$$\pro{W(\tau)\leq\eps\, ; \,W>0;\tau<\infty} \leq C\left(\eps + \eps^{-1}\underset{n}{\lim\inf}~\esp{e^{-r_\star (\tau\wedge n)/2}}\right).$$
	
	Noticing that
	$$e^{-r_\star (\tau\wedge n)/2} \leq \uno{\tau<\infty}e^{-r_\star\tau/2} + e^{-r_\star n/2},$$
	we have
	$$\pro{W(\tau)\leq\eps\, ; \,W>0\, ; \,\tau<\infty} \leq C\left(\eps + \eps^{-1}\esp{\uno{\tau<\infty} e^{-r_\star \tau/2}}\right).$$
	
	In addition, to control the second term of the sum in~\eqref{controlwtau}, we recall that  the distribution of~$W$ is
	$$\frac{d_\star}{b_\star}\delta_0 + \frac{r_\star}{b_\star}{\cal E}(\frac{r_\star}{b_\star}).$$
	Hence
	$$\pro{0<W\leq\eps} = \frac{r_\star}{b_\star}\left(1-e^{-\frac{r_\star}{b_\star} \eps}\right)\leq (\frac{r_\star}{b_\star})^2\eps.$$
	
	Recalling~\eqref{controlwtau}, the result of the lemma is proved.
\end{proof}

\medskip
\begin{lemme}\label{infw}
With the same notation as in Lemma~\ref{propbranchement}, for any~$0<\eps<1,$
$$\pro{\underset{t\geq 0}{\inf}~W(t) \leq \eps\, ; \,W>0}\leq C\eps^{1/4},$$
for some positive constant~$C$ independent of~$\eps.$
\end{lemme}

\begin{proof}
    Let us introduce
    $$s_\eps := \frac1{r_\star}\log\left(1/\eps\right).$$
    
    On the event $\{W>0\},$ we have $Z(t)\geq 1$ for all $t\geq 0$ and further,  for $t<s_\eps, $ $e^{-r_\star t}>\eps.$ As a consequence
    $$\left\{\underset{t\geq 0}{\inf}~W(t) \leq \eps\, ; \,W>0\right\}\subseteq \left\{\underset{t\geq s_\eps}{\inf}~W(t) \leq \eps\, ; \,W>0\right\}.$$
    
    For some $\lambda>1$ whose value will be fixed later, let us denote
    $$S_\eps := \inf\left\{t\geq s_\eps~:~W(t) \leq \eps^{1/\lambda}\right\}.$$
    
    Since $\lambda>1$ and $0<\eps<1,$ we have $\eps < \eps^{1/\lambda},$ hence
    $$\left\{\underset{t\geq s_\eps}{\inf}~W(t) \leq \eps\right\}\subseteq\left\{S_\eps<\infty\right\}\subseteq\left\{W(S_\eps)\leq\eps^{1/\lambda}\right\}.$$
    
    Note that the last event above is trivially satisfied only if~$S_\eps<\infty.$
    
    Then, thanks to Lemma~\ref{propbranchement},
    $$\pro{W(S_\eps)\leq \eps^{1/\lambda}\, ; \,W>0} \leq C\left(\eps^{1/\lambda} + \eps^{-1/\lambda}\esp{\uno{S_\eps<\infty}e^{-r_\star S_\eps/2}}\right).$$
    
    Recall that, by definition of $S_\eps,$ $S_\eps\geq s_\eps,$ and so
    $$\esp{\uno{S_\eps<\infty}e^{-r_\star S_\eps/2}} \leq e^{-r_\star s_\eps/2} = \eps^{1/2},$$
    implying that
    $$\pro{W(S_\eps)\leq \eps^{1/\lambda}\, ; \,W>0} \leq C\left(\eps^{1/\lambda} + \eps^{1/2-1/\lambda}\right).$$
    
    Finally, choosing $\lambda := 4$ to optimize the bound above proves the result.
\end{proof}

\subsection{Perturbation of stable  dynamical system}

We provide a result which is strongly inspired from Lemmas~3.1 and~3.2 in \cite{prodhomme_strong_2020}.
We consider an exponential stable dynamical system  : $y'\leq -c y $, with $c>0$. We add
a source term, which acts as a perturbation given by $h$. The following result allows to control the value of the dynamical system in terms of fluctuations of the source term $h$ on a unit time interval.
\begin{lemme}\label{lemmeadrien}
	Assume that there are three measurable and locally bounded functions~$y,h,\phi$ with $h(0)=0$ and  satisfying for all~$t\geq 0,$
	$$y(t) = y(0) + h(t) + \int_0^t y(s)\phi(s)ds.$$
	Suppose in  addition that $\phi$ is bounded, and upper-bounded by some negative number. 
 Then, for all $t\geq 0,$
	$$\underset{s\leq t}{\sup}~|y(s)|\leq \Gamma~ \left(|y(0)|\vee\underset{s\leq t}{\sup}\left|h(s)-h(\lfloor s\rfloor)\rr|\rr),$$
	where ~$\Gamma=(1+||\phi||_{\infty})/(1-\exp(\sup \phi ))$.
\end{lemme}

\begin{proof}
To begin with, we introduce
		$$\Phi(t,s) = e^{\int_s^t \phi(r)dr} \leq e^{-C(t-s)},$$
		for some positive constant~$C := -\sup\phi>0.$
		Then
		\begin{align*}
		\Phi(t,t) =& 1 \ ;\ \Phi(t,s) \Phi(s,u) = \Phi(t,u),\\  \partial_t\Phi(t,s) =& \Phi(t,s)\phi(t) \ ; \ 
		\partial_s\Phi(t,s) = -\Phi(t,s)\phi(s).
		\end{align*}
		By a standard  constant variation argument, one can write
		$$y(t) = \Phi(t,0)y(0) + h(t) +\int_{0}^t \Phi(t,s) h(s) \phi(s) ds.$$

Then  it results the following decomposition:
		\begin{align*}
		y(t)=& \Phi(t,0)y(0) + \sum_{j=1}^{\lfloor t\rfloor} \Phi(t,j)\left(h(j) - h(j-1) +\int_{j-1}^j \Phi(j,s)\phi(s)\left(h(s) - h(j-1)\rr)ds\rr)\\
		&+h(t) - h(\lfloor t\rfloor) + \int_{\lfloor t\rfloor}^t \Phi(t,s)\phi(s)\left(h(s) - h(\lfloor t\rfloor)\rr)ds.
		\end{align*}
		Using the upperbound of $\phi$, we obtain
		\begin{align*}
		|y(t)|\leq& e^{-Ct}|y(0)| + \sum_{j=1}^{\lfloor t\rfloor} e^{-C(t-j)}\left(\underset{j-1\leq s\leq j}{\sup} \left|h(s) - h(j-1)\rr|\Gamma'\rr) + \underset{\lfloor t\rfloor\leq s\leq t}{\sup} \left|h(s) - h(\lfloor t\rfloor)\rr|\Gamma',
		\end{align*}
		with $\Gamma' := 1 + ||\phi||_{\infty}.$  The result is proved by choosing
		$\Gamma = \Gamma'\sum_{j=0}^\infty e^{-Cj} = \Gamma'/(1-e^{-C}).$
\end{proof}

\bigskip

\noindent \textbf{Acknowledgements}:
\begin{itemize}
\item The authors are very grateful to Pierre Collet who has clarified some issue  involved in the control of hitting values of the dynamical system started close to the equilibrium. They also thank very warmly the two referees of the first version of this work, whose many comments  led to a major improvement of the paper.
\item This work has been supported by the Chair ``Mod\'elisation Math\'ematique et Biodiversit\'e of Veolia Environnement-Ecole Polytechnique-Museum national d'Histoire naturelle-Fondation X  and by  ITMO Cancer of Aviesan within the framework of the 2021-2023 Cancer Control Strategy, on funds administrated by Inserm. 
\item Funded by the European Union (ERC, SINGER, 101054787). Views and opinions expressed are however those of the author(s) only and do not necessarily reflect those of the European Union or the European Research Council. Neither the European Union nor the granting authority can be held responsible for them.
\end{itemize}

{\footnotesize
\providecommand{\noopsort}[1]{}\providecommand{\noopsort}[1]{}\providecommand{\noopsort}[1]{}\providecommand{\noopsort}[1]{}

}

\end{document}